\documentclass{article}

\usepackage{extarrows}
\usepackage{amsthm}
\usepackage{latexsym}
\usepackage[american]{babel}
\usepackage{times}
\usepackage{tikz}
\usepackage{tikz-cd}
\usepackage{comment}
\usepackage{url}
\usepackage{mathrsfs}
\usepackage{amsmath,amstext,amsthm,amssymb,amsfonts,amscd}
\usepackage{mathrsfs}
\usepackage[colorlinks=true]{hyperref}
\usepackage{graphicx}
\usepackage[T1]{fontenc}
\usepackage[latin1]{inputenc}
\usepackage{color,tocvsec2}
\usepackage{typearea}
\usepackage{charter}
\usepackage{enumerate}
\usepackage{lscape}
\usepackage[all]{xy}
\usepackage[normalem]{ulem}

\RequirePackage[T1]{fontenc}
\title{Residue currents of cohesive modules and the generalized Poincar\'{e}-Lelong formula on complex manifolds}
\author{ Zhaoting Wei\\ Texas A\&M University-Commerce\\ zhaoting.wei@tamuc.edu}
\date{}

\begin{document}
\newcommand{\End}{\text{End}}
\newcommand{\HBC}{H_\text{BC}}
\newcommand{\chBC}{\text{ch}_\text{BC}}
\newcommand{\Hom}{\text{Hom}}
\newcommand{\ad}{\text{ad}}
\newcommand{\id}{\text{id}}
\newcommand{\str}{\text{Tr}_{\text{s}}}
\newcommand{\pr}{\text{pr}}
\newcommand{\Ad}{\text{Ad}}
\newcommand{\Vect}{\text{Vect}}
\newcommand{\CoAd}{\text{CoAd}}
\newcommand{\coad}{\text{coad}}
\newcommand{\Pol}{\text{Pol}}
\newcommand{\Cl}{\text{Cl}}
\newcommand{\As}{\text{As}}
\newcommand{\ch}{\text{ch}}
\newcommand{\Lied}{\text{L}}
\newcommand{\Ext}{\text{Ext}}
\newcommand{\ZZ}{\mathbb{Z}/2\mathbb{Z}}
\newcommand{\bHom}{\mathbb{H}\text{om}}
\newcommand{\bCone}{\underline{\mathbb{C}\text{one}}}
\newcommand{\im}{\text{im }}
\newcommand{\Acyc}{\text{Acyc}}
\newcommand{\Cinfty}{C^{\infty}}
\newcommand{\gb}{\text{gb}}
\newcommand{\diff}{\text{d}}
\newcommand{\Diff}{\text{Diff}}
\newcommand{\dpar}{\partial}
\newcommand{\dbar}{\overline{\partial}}
\newcommand{\pprime}{\prime\prime}
\newcommand{\coh}{\text{coh}}
\newcommand{\KS}{\text{KS}}
\newcommand{\ku}{\text{ku}}
\newcommand{\R}{\mathbb{R}}
\newcommand{\C}{\mathbb{C}}
\newcommand{\codim}{\text{codim}}
\newcommand{\supp}{\text{supp}}

\newtheorem{thm}{Theorem}[section]
\newtheorem{lemma}[thm]{Lemma}
\newtheorem{prop}[thm]{Proposition}
\newtheorem{coro}[thm]{Corollary}
\newtheorem{defi}{Definition}[section]
\newtheorem{eg}{Example}[section]
\newtheorem{rmk}{Remark}[section]

\numberwithin{equation}{section}

\maketitle

\begin{abstract}
Cohesive module provides a tool to study  coherent sheaves on complex manifolds by global analytic methods. In this paper we develop the theory of residue currents for cohesive modules on complex manifolds. In particular we prove that they have the duality principle and satisfy the comparison formula. As an application, we prove a generalized version of the Poincar\'{e}-Lelong formula for cohesive modules, which applies to coherent sheaves without globally defined locally free resolutions.


\end{abstract}

\section{Introduction}
Let $X$ be a complex manifold, and let
\begin{equation}\label{eq:bundle-complex}
	0 \longrightarrow E^{-N} \longrightarrow
	\dots \longrightarrow E^{-1} \longrightarrow
	E^0 \longrightarrow 0
\end{equation}
be a generically exact complex of holomorphic vector bundles over $X$. In \cite{andersson2007residue}, Andersson and Wulcan  constructed an $(\End E)$-valued current $R^E$, which is called the \emph{residue current} associated with the complex $E^\bullet$. The main result that they proved is the \emph{duality principle}, which claims that if the corresponding complex of locally free sheaves is exact at each level $r < 0$, then $R^E$ has the property that a holomorphic section $\phi$ of $E^0$ belongs to $\im \left(E^{-1} \to E^0 \right)$ if and only if $R^E \phi = 0$.

Andersson's and Wulcan's construction is a generalization of the residue current of  a holomorphic function in \cite{herrera1971residues} and the Coleff--Herrera current of a tuple of holomorphic functions in \cite{herrera1978courants}. These development has led to many results in commutative algebra and complex geometry.  Suppose that \eqref{eq:bundle-complex}, as a complex of locally free $\mathcal{O}_X$-modules, is a locally free resolution of a coherent $\mathcal{O}_X$-module $\mathfrak{F}$. The current $R^E$ is then considered as a current representation of the sheaf $\mathfrak{F}$.  In \cite{larkang2018residue} and \cite{larkang2021residue}, L\"{a}rk\"{a}ng and Wulcan proved that if $\mathfrak{F}$ has pure codimension $p\geq 1$, then  we have
\begin{equation}\label{eq: Poincare-Lelong in Larkang2 introduction}
\frac{1}{(2\pi i)^p p!}\text{Tr}(\nabla^{E^{\bullet}}(v_{-1})\ldots \nabla^{E^{\bullet}}(v_{-p}) R^E_{0\to -p})=[\mathfrak{F}]
\end{equation}
where $R^E_{0\to -p}$ is the component of $R^E$ that maps $E^0$ to $E^{-p}$, and $[\mathfrak{F}]$ is the \emph{cycle} associated to $\mathfrak{F}$.  Notice that \eqref{eq: Poincare-Lelong in Larkang2 introduction} is a generalization of the classical \emph{Poincar\'{e}-Lelong formula}
\begin{equation}
\frac{1}{2\pi i}\dbar\dpar \log|f|^2=[Z_f].
\end{equation}

Given a coherent $\mathcal{O}_X$-module $\mathfrak{F}$ on a complex manifold $X$, although locally free resolution of $\mathfrak{F}$ always exists \emph{locally}, it may not exist \emph{globally}. See \cite[Corollary~A.5]{voisin2002counterexample} for an example of a coherent $\mathcal{O}_X$-module which does not admit a globally defined locally free resolution. Thus unless one is restricted to the setting where global resolutions of locally free sheaves always exist, e.g. $X$ is a projective manifold, it is not always possible to use the residue current introduced in \cite{andersson2007residue} to study the global properties of $\mathfrak{F}$.

In \cite{block2010duality} Block introduced the concept of \emph{cohesive modules}.  For a complex manifold $X$, a cohesive module $\mathcal{E}$ on $X$ consists of a cochain complex of $C^{\infty}$ vector bundles $E^{\bullet}$ together with a flat $\dbar$-superconnection $A^{E^{\bullet}\prime\prime}$. Cohesive modules on $X$ form a dg-category $B(X)$.  Block proved in \cite{block2010duality} that if $X$ is compact, then $B(X)$ gives a dg-enhancement of $D^b_{\text{coh}}(X)$, the bounded derived category of coherent sheaves on $X$.  Later \cite{chuang2021maurer} generalized the result in \cite{block2010duality} to the case that $X$ is non-compact with a slightly more restricted definition of coherent sheaves. According to \cite{block2010duality} and \cite{chuang2021maurer}, a coherent sheaf $\mathfrak{F}$ on a complex manifold always admits a globally defined \emph{cohesive resolution}. See Section \ref{section: cohesive modules} for a quick review of cohesive modules and results in \cite{block2010duality} and \cite{chuang2021maurer}.

Block's result makes it possible to apply global analytic method to the study of coherent sheaves on general non-projective complex manifolds. For one application see \cite{bismut2023coherent}, in which Bismut, Shen, and the author proved the Riemann-Roch-Grothendieck theorem for coherent sheaves on complex manifolds, by bringing together Block's result, local index theory, and hypoelliptic operators. 

In the current paper we construct and study the residue current $R^{\mathcal{E}}$ of a cohesive module $\mathcal{E}$. We show that the residue current of a cohesive module has duality principle as expected. See Theorem  \ref{thm: duality principle sheaf case} for details.

One of the advantages of the dg-category of cohesive modules $B(X)$ over the derived category $D^b_{\text{coh}}(X)$ is that any quasi-isomorphism in $B(X)$ has a homotopy inverse. In this paper we give a \emph{comparison formula} for residue currents of cohesive modules, which gives the compatibility of residue currents with morphisms between cohesive modules. In particular we show that, under homotopy invertible morphisms, residue currents are invariant modulo coboundary elements. See Corollary \ref{coro: homotopic comparison formula} for details.

As an application, we prove the generalized \emph{Poincar\'{e}-Lelong formula} in the framework of cohesive modules. In more details, let $\mathfrak{F}$ be a coherent sheaf of pure codimension $p\geq 1$ and $\mathcal{E}$ be a cohesive resolution of $\mathfrak{F}$, then we  have the following equality of currents:
\begin{equation}\label{eq: Poincare-Lelong for cohesive modules 1 introduction}
\frac{1}{(2\pi i)^p p!}\str((\nabla^{E^{\bullet}}(v_0))^p R^{\mathcal{E}})=[\mathfrak{F}],
\end{equation} 
where $\str$ denotes the \emph{supertrace}.
Here we do not assume the global existence of locally free resolutions, hence the result applies to general complex manifolds, projective or not. See Theorem \ref{thm: Poincare-Lelong for cohesive modules} and Corollary \ref{coro: Poincare-Lelong non-pure} for details.

This paper is organized as follows: 
In Section \ref{section: cohesive modules} we review cohesive modules on complex manifolds. In Section \ref{section: pseudomeromorphic currents} we review pseudomeromorphic and almost semimeromorphic currents on complex manifolds. In Section \ref{section: residue currents of cohesive modules} we define residue currents for cohesive modules and study their initial properties. In Section \ref{section: vanishing of residue current} we study the vanishing property of residue currents, which leads to the duality principle as in Theorem  \ref{thm: duality principle sheaf case}. In Section \ref{section: comparison formula} we give the comparison formula of residue currents under morphisms between cohesive modules. Finally in Section \ref{section: Poincare-Lelong} we give and prove the generalized Poincar\'{e}-Lelong formula in Theorem \ref{thm: Poincare-Lelong for cohesive modules} and Corollary \ref{coro: Poincare-Lelong non-pure}.

\subsection*{Related works}
\emph{Twisting cochain}, which was introduced by Toledo and Tong in \cite{toledo1978duality}, is another approach to the global study of coherent sheaves on non-projective complex manifolds. Actually a twisting cochain consists of \v{C}ech style higher structures, while a cohesive module consists of Dolbeault style higher structures. In \cite{johansson2021explicit} and \cite{Johansson2023residue},  Johansson and L\"{a}rk\"{a}ng developed the theory of  residue currents for twisting cochains. In \cite{Johansson2023residue}  Johansson also proved the duality principle and comparison formula for residue currents of twisting cochains. A large part of the current paper can be considered as a parallel work to \cite{johansson2021explicit} and \cite{Johansson2023residue} and much of the inspirations come from there. Although  the residue currents defined in the current paper and those defined in \cite{johansson2021explicit} and \cite{Johansson2023residue} apparently live in different spaces, we expect deep relationship between them.

We also notice that in \cite{han2024characteristic} Han introduced  characteristic currents on cohesive modules. Notice that for complexes of holomorphic vector bundles, residue currents and characteristic currents are closed related as shown in \cite{larkang2022chern}. It will be interesting to find similar relation between the constructions in  \cite{han2024characteristic} and in this paper.

\section*{Acknowledgment}
The author wants to Zhizhang Xie and Jinmin Wang for very inspiring discussions. He also wants to thank Richard L\"{a}rk\"{a}ng for kindly answering questions on residue currents for twisting cochains.

\section{A review of cohesive modules on complex manifolds}\label{section: cohesive modules}
\subsection{The definition of cohesive modules}\label{subsection: cohesive modules}
We first fix some notations. Let $X$ be a  complex manifold of dimension $n$. Let $TX$ and $\overline{TX}$ be the holomorphic and antiholomorphic tangent bundle. Let  $T_{\mathbb{R}}X$ be the corresponding real tangent bundle and $T_{\mathbb{C}}X=T_{\mathbb{R}}X\otimes_{\mathbb{R}}\mathbb{C}$ be its complexification. We have the decomposition $T_{\mathbb{C}}X=TX\bigoplus \overline{TX}$. Let $\Omega^{p,q}_X$ be the sheaf of smooth $(p,q)$-forms on $X$.

The concept of cohesive modules is introduced by Block in \cite{block2010duality}. 

\begin{defi}\label{defi: cohesive module}
Let $X$ be a complex manifold. A \emph{cohesive module} on $X$ is a bounded, finite rank, $\mathbb{Z}$-graded, $C^{\infty}$-vector bundle $E^{\bullet}$ on $X$ together with a superconnection with total degree $1$
$$
A^{E^{\bullet}\prime\prime}: \wedge^{\bullet}\overline{T^{*}X}  \times  E^{\bullet}\to \wedge^{\bullet}\overline{T^{*}X}  \times E^{\bullet}
$$
such that $A^{E^{\bullet}\prime\prime}\circ A^{E^{\bullet}\prime\prime}=0$.

In more details, $A^{E^{\bullet}\prime\prime}$ decomposes into
\begin{equation}\label{eq: decomposition of anti super conn}
A^{E^{\bullet}\prime\prime}=v_0+\nabla^{E^{\bullet}\prime\prime}+v_2+\ldots
\end{equation}
where 
$$
\nabla^{E^{\bullet}\prime\prime}: E^{\bullet}\to \overline{T^{*}X}  \times E^{\bullet}
$$
 is a  $\dbar$-connection, and for $i\neq 1$
\begin{equation}\label{eq: cohesive module}
v_i\in C^{\infty}(X,\wedge^{i}\overline{T^{*}X}  \hat{\otimes}  \End^{1-i}(E^{\bullet}))
\end{equation}
is  $C^{\infty}(X)$-linear. Here $\hat{\otimes} $ denotes the graded tensor product. The equation  $A^{E^{\bullet}\prime\prime}\circ A^{E^{\bullet}\prime\prime}=0$ decomposes into
\begin{equation}\label{eq: A flat decomposed}
\begin{split}
&v_0^2=0,\\
&\nabla^{E^{\bullet}\prime\prime}(v_0)=0,\\
&(\nabla^{E^{\bullet}\prime\prime})^2+[v_0,v_2]=0,\\
&\ldots
\end{split}
\end{equation}

Cohesive modules on $X$ forms a dg-category  denoted by $B(X)$. In more details, let $\mathcal{E}=(E^{\bullet}, A^{E^{\bullet}\prime\prime})$ and $\mathcal{F}=(F^{\bullet}, A^{F^{\bullet}\prime\prime})$ be two cohesive modules on $X$ where
$$
A^{E^{\bullet}\prime\prime}=v_0+\nabla^{E^{\bullet}\prime\prime}+v_2+\ldots
$$ 
and
$$
A^{F^{\bullet}\prime\prime}=u_0+\nabla^{F^{\bullet}\prime\prime}+u_2+\ldots
$$ 
A morphism $\phi:\mathcal{E}\to \mathcal{F}$ of degree $k$ is given by
\begin{equation}\label{eq: decomposition of morphism}
\phi=\phi_0+\phi_1+\ldots
\end{equation}
where
$$
\phi_i\in C^{\infty}(X,\wedge^{i}\overline{T^{*}X}  \hat{\otimes}  \Hom^{k-i}( E^{\bullet},F^{\bullet}))
$$
is   $C^{\infty}(X)$-linear.

For 
$$
\phi=\alpha \hat{\otimes} u\in  C^{\infty}(X,\wedge^{i}\overline{T^{*}X}  \hat{\otimes}  \Hom^{k-i}( E^{\bullet},F^{\bullet}))$$ 
and 
$$
\psi=\beta\hat{\otimes} v\in C^{\infty}(X,\wedge^{j}\overline{T^{*}X}  \hat{\otimes}  \Hom^{l-j}( F^{\bullet},G^{\bullet})),
$$ 
their composition $\psi\phi$ is defined as
\begin{equation}\label{eq: composition of morphisms in B(X)}
\psi \phi:=(-1)^{(l-j)i}\beta\alpha\hat{\otimes} vu\in C^{\infty}(X,\wedge^{i+j}\overline{T^{*}X}  \hat{\otimes}  \Hom^{k+l-i-j}( E^{\bullet},G^{\bullet}))
\end{equation}

 The differential of $\phi$ is given by 
\begin{equation}\label{eq: differential in B(X)}
D^{\mathcal{E},\mathcal{F}}\phi=A^{F^{\bullet}\prime\prime}\phi-(-1)^k\phi A^{E^{\bullet}\prime\prime}.
\end{equation}
More explicitly, the $l$th component  of $d\phi$ is 
$$
(D^{\mathcal{E},\mathcal{F}}\phi)_l\in C^{\infty}(X,\wedge^{l}\overline{T^{*}X}  \hat{\otimes}  \Hom^{k-l+1}( E^{\bullet},F^{\bullet}))
$$
which is given by
\begin{equation}\label{eq: differential in B(X) degree l}
(D^{\mathcal{E},\mathcal{F}}\phi)_l=\sum_{i\neq 1}\big(u_i\phi_{l-i}-(-1)^k\phi_{l-i}v_i\big)+\nabla^{F^{\bullet}\prime\prime}\phi_{l-1}-(-1)^k\phi_{l-1}\nabla^{E^{\bullet}\prime\prime}.
\end{equation}
\end{defi}

\begin{rmk}
In \cite{bismut2023coherent} cohesive modules are called \emph{antiholomorphic superconnections}.
\end{rmk}

We can define mapping cones and shift in $B(X)$. For a degree zero closed map $\phi: \mathcal{E}\to \mathcal{F}$ where $\mathcal{E}= (E^{\bullet}, A^{E^{\bullet}\prime\prime})$ and $ \mathcal{F}=(F^{\bullet}, A^{F^{\bullet}\prime\prime})$, its mapping cone $(C^{\bullet}, A^{C^{\bullet}\prime\prime}_{\phi})$ is defined by
\begin{equation}
C^{n}=E^{n+1}\bigoplus F^n
\end{equation}
and
\begin{equation}
A^{C^{\bullet}\prime\prime}=\begin{bmatrix}A^{E^{\bullet}\prime\prime}&0\\ \phi(-1)^{\deg(\cdot)}& A^{F^{\bullet}\prime\prime}\end{bmatrix}.
\end{equation}
The shift of $\mathcal{E}$ is $\mathcal{E}[1]$ where
\begin{equation}
E[1]^{n}=E^{n+1}
\end{equation}
and
$$
A^{E^{\bullet}\prime\prime}[1]=A^{E^{\bullet}\prime\prime}(-1)^{\deg(\cdot)}.
$$
It is clear that they give $B(X)$ a pre-triangulated structure hence its homotopy category $\underline{B}(X)$ is a triangulated category.

For later purpose, we recall the following definition

\begin{defi}\label{defi: gauge and homotopy equivalence}
A degree $0$ closed morphism $\phi$ between cohesive modules $\mathcal{E}$  and $\mathcal{F}$  is called a \emph{gauge equivalence} if it admits an inverse in $B(X)$, i.e. if there exists a degree $0$ closed morphism $\psi$  from $\mathcal{F}$ to $\mathcal{E}$ such that $\psi\circ \phi=\id_{\mathcal{E}}$ and $\phi\circ \psi=\id_{\mathcal{F}}$.

A degree $0$ closed morphism  $\phi$ is   called a \emph{homotopy equivalence} if it induces an isomorphism in the homotopy category $\underline{B}(X)$. 
\end{defi}

We will need the  following results.

\begin{prop}\label{prop: gauge and homotopy equiv degree 0}
A degree $0$ closed morphism $\phi$ between cohesive modules $\mathcal{E}= (E^{\bullet}, A^{E^{\bullet}\prime\prime})$ and $\mathcal{F}=(F^{\bullet}, A^{F^{\bullet}\prime\prime})$  is  a gauge equivalence if and only if its degree $0$ component $\phi^0:  (E^{\bullet},v_0)\to (F^{\bullet},u_0)$ is invertible at each degree. It is a homotopy equivalence if and only if $\phi^0$ is a quasi-isomorphism of cochain complexes.
\end{prop}
\begin{proof}
The first claim is obvious. The second claim is proved in \cite[Proposition 2.9]{block2010duality} or \cite[Proposition 6.4.1]{bismut2023coherent}.
\end{proof}

\subsection{Pull-backs of cohesive modules}\label{subsection: pull backs}
Let $f: X\to Y$ be a holomorphic map between complex manifolds.

\begin{lemma}\label{lemma: pull-back of coherent is coherent}
 Let  $\mathcal{E}$ be a bounded complexes of $\mathcal{O}_Y$-modules with globally bounded coherent cohomologies. Then
\begin{equation}
f^*\mathcal{E}:=f^{-1}\mathcal{E}\otimes_{f^{-1}\mathcal{O}_Y}\mathcal{O}_X
\end{equation}
is a bounded complexes of $\mathcal{O}_X$-modules with globally bounded coherent cohomologies. 
\end{lemma}
\begin{proof}
The coherence is given by \cite[Section 1.2.6]{grauert1984coherent}. The global boundedness is clear from the definition and the fact that $f^*\mathcal{O}_Y^N=\mathcal{O}_X^N$.
\end{proof}

Hence we can define
the left derived functor 
\begin{equation}\label{eq: left derived pull-back functor}
Lf^*: D^{\gb}_{\coh}(Y)\to D^{\gb}_{\coh}(X).
\end{equation}

\begin{lemma}\label{lemma: derived pull-back flat modules}
If $\mathcal{E}\in D^{\gb}_{\coh}(Y)$ is a bounded complex of flat $\mathcal{O}_Y$-modules, then we have
\begin{equation}\label{eq: derived pull-back of flat modules}
Lf^*\mathcal{E}=f^*\mathcal{E}.
\end{equation}
\end{lemma}
\begin{proof}
By \cite[\href{https://stacks.math.columbia.edu/tag/064K}{Tag 064K}]{stacks-project}, any bounded complex of flat modules is K-flat. Then the lemma is a consequence of \cite[\href{https://stacks.math.columbia.edu/tag/06YJ}{Tag 06YJ}]{stacks-project}.
\end{proof}

We can also define the pull-backs of cohesive modules. Notice that $f^*$ maps $T^{*}Y$ to $T^{*}X$, hence $\wedge\overline{T^{*}X}$ is a $\wedge \overline{f^*T^{*}Y}$-module.

\begin{defi}\label{defi: pull-back of cohesive modules}
Let $\mathcal{E}=(E^{\bullet}, A^{E^{\bullet}\prime\prime})\in B(Y)$ be a cohesive module on $Y$. We define its pull-back $f^*_b\mathcal{E}$ to be 
$$
(f^*E^{\bullet}, f^*A^{E^{\bullet}\prime\prime})
$$
where $f^*E^{\bullet}$ is the pull-back graded vector bundle and $f^*A^{E^{\bullet}\prime\prime}$ is the pull-back superconnection. In more details, if
$$
A^{E^{\bullet}\prime\prime}=v_0+\nabla^{E^{\bullet}\prime\prime}+v_2+\ldots
$$ 
is the decomposition in \eqref{eq: decomposition of anti super conn}. Then 
\begin{equation}\label{eq: decomposition of the pull-back connection}
f^*A^{E^{\bullet}\prime\prime}=f^*v_0+f^*\nabla^{E^{\bullet}\prime\prime}+f^*v_2+\ldots
\end{equation}
where $f^*\nabla^{E^{\bullet}\prime\prime}$ is the pull-back connection on $f^*E^{\bullet}$, and $f^*v_i$ is the pull-back form valued in $\wedge^{i}\overline{T^{*}X}  \hat{\otimes}  \End^{1-i}(E^{\bullet})$.

If $\phi: \mathcal{E}\to \mathcal{F}$ is a morphism, then we have the pull-back morphism $f^*_b\phi: f^*_b\mathcal{E}\to f^*_b\mathcal{F}$ defined by pulling back each component of $\phi$.
\end{defi}

In particular, if $i: X\hookrightarrow Y$ is an open or closed embedding, then we denote $i^*_b\mathcal{E}$ by $\mathcal{E}|_X$.

It is easy to see that $f^*_b$ defines a dg-functor $B(Y)\to B(X)$ hence we get the functor $\underline{f^*_b}: \underline{B}(Y)\to \underline{B}(X)$.

The following proposition implies that a cohesive module is locally the same as a cochain complex of holomorphic vector bundles.

\begin{prop}\label{prop: gauge equivalence to flat}
 For a cohesive module $\mathcal{E}= (E^{\bullet}, A^{E^{\bullet}\prime\prime})$ on $X$. For any $x\in X$, there exists an open neighborhood $V$ of $x$ and a flat $\dbar$-connection $ \overline{\nabla}^{E^{\bullet}|_V\prime\prime}$ on $E^{\bullet}|_V$ such that
\begin{enumerate}
\item $ \overline{\nabla}^{E^{\bullet}|_V\prime\prime}(v_0)=0$, i.e. $(E^{\bullet}|_V, v_0+ \overline{\nabla}^{E^{\bullet}|_V\prime\prime})$ is a cohesive module on $V$ with $v_i=0$ for all $i\geq 2$;
\item There exists a gauge equivalence $J: (E^{\bullet}, A^{E^{\bullet}\prime\prime})|_V\overset{\sim}{\to} (E^{\bullet}|_V, v_0+ \overline{\nabla}^{E^{\bullet}|_V\prime\prime})$.
\end{enumerate}
\end{prop}
\begin{proof}
See \cite[Lemma 4.5]{block2010duality} or \cite[Theorem 5.2.1]{bismut2023coherent}.
\end{proof}

\begin{rmk}
Notice that the gauge equivalence $J$ in Proposition \ref{prop: gauge equivalence to flat} does not change the map $v_0: E^{\bullet}\to E^{\bullet+1}$.
\end{rmk}

\subsection{Coherent sheaves and an equivalent of categories}\label{subsection: cohesive and coherent}
 Cohesive modules are closely related to coherent sheaves on $X$. Let $\mathcal{O}_X$ be the sheaf of holomorphic functions. We call a sheaf of $\mathcal{O}_X$-modules $\mathfrak{E}$ \emph{coherent} if it satisfies the following two conditions
\begin{enumerate}
\item $\mathfrak{E}$ is of finite type over $\mathcal{O}_X$, that is, every point in $X$ has an open neighborhood $U$ in $X$ such that there is a surjective morphism $\mathcal {O}_{X}^{n}|_{U}\twoheadrightarrow \mathcal {F}|_{U}$ for some natural number $n$;
\item for \emph{any} open set $U\subseteq X$, \emph{any} natural number $n$, and \emph{any} morphism $\varphi :\mathcal {O}_{X}^{n}|_{U}\to \mathcal {F}|_{U}$ of $\mathcal {O}_{X}$-modules, the kernel of $\varphi$  is of finite type.
\end{enumerate}

Let $D^b_{\coh}(X)$ be the derived category of bounded complexes of $\mathcal{O}_X$-modules with coherent cohomologies.

\begin{thm}\label{thm: equiv of cats}[\cite[Theorem 4.3]{block2010duality}, \cite[Theorem 6.5.1]{bismut2023coherent}]
If $X$ is a compact complex manifold, then  there exists an equivalence $\underline{F}_X: \underline{B}(X)\overset{\sim}{\to}D^b_{\coh}(X)$ as triangulated categories. Here $\underline{B}(X)$ is the homotopy category of $B(X)$.
\end{thm}

In \cite{chuang2021maurer} the result of Theorem \ref{thm: equiv of cats} is generalized to noncompact complex manifold. Recall that a coherent sheaf $\mathfrak{E}$ is called \emph{globally bounded} if  there exists an open covering $U_i$ of $X$ and  integers $a<b$ and $N>0$ such that on each $U_i$ there exists a bounded complex of finitely generated locally free $\mathcal{O}_X$-modules $\mathcal{S}^{\bullet}_i$ which is concentrated in degrees $[a, b]$ and each $\mathcal{S}^j_i$ has rank $\leq N$, together with  a quasi-isomorphism $\mathcal{S}^{\bullet}_i\to \mathfrak{E}^{\bullet}|_{U_i}$.

Let $D^{\gb}_{\coh}(X)$ be the full subcategory of $D^b_{\coh}(X)$ whose objects are bounded complexes of $\mathcal{O}_X$-modules with globally bounded coherent cohomologies. When $X$ is compact, it is clear that $D^{\gb}_{\coh}(X)$ coincides with $D^b_{\coh}(X)$. Moreover if $X$ is compact and $\mathfrak{F}\in D^b_{\coh}(X)$, then for any open subset $V\subset X$, it is clear that the restriction $\mathfrak{F}|_V$ is in $D^{\gb}_{\coh}(V)$. 

\begin{rmk}
In this paper when we talk about complexes of sheaves with coherent cohomologies, we always assume it is globally bounded.
\end{rmk}

\begin{thm}[\cite{chuang2021maurer} Theorem 8.3]\label{thm: equiv of cats noncompact}
If $X$ is a  complex manifold, then there exists an equivalence  $\underline{F}_X: \underline{B}(X)\overset{\sim}{\to}D^{\gb}_{\coh}(X)$ as triangulated categories.
\end{thm}

For an object $\mathfrak{F}\in D^{\gb}_{\coh}(X)$, if $\mathcal{E}\in B(X)$ is a cohesive module such that  $\underline{F}_X(\mathcal{E})$ is quasi-isomorphic to $\mathfrak{F}$, then we call $\mathcal{E}$ a \emph{cohesive resolution} of $\mathfrak{F}$. In particular we can talk about cohesive resolutions of a single coherent sheaf, considered as a complex of sheaves concentrated in degree $0$. Theorem \ref{thm: equiv of cats noncompact} implies that cohesive resolutions always exist.
 
For later applications we give the construction of the functor $\underline{F}_X$ here. For a cohesive module $\mathcal{E}= (E^{\bullet}, A^{E^{\bullet}\prime\prime})$, we define $\underline{F}_X(\mathcal{E})$ to be the cochain complex $(\mathfrak{E}^{\bullet},d)$, where the sheaf $\mathfrak{E}^{\bullet}$ is given by
\begin{equation}\label{eq: sheafify of cohesive module}
\mathfrak{E}^n(U):=\bigoplus_{p+q=n}\Gamma(X, \Omega^{0,p}\hat{\otimes} E^q)
\end{equation}
and $d: \mathfrak{E}^n\to \mathfrak{E}^{n+1}$ is exactly $A^{E^{\bullet}\prime\prime}$. 

The following results are part of Theorem \ref{thm: equiv of cats noncompact}. We state them for later for the convenience of later applications.
\begin{prop}\label{prop: sheafify has coherent cohomologies}
The cochain complex cochain complex $(\mathfrak{E}^{\bullet},d)$ above has (globally bounded) coherent cohomologies.
\end{prop}
\begin{proof}
It is a direct consequence of Proposition \ref{prop: gauge equivalence to flat}.
\end{proof}

\begin{prop}\label{prop: sheafify morphisms}
Any closed degree $0$ morphism $\phi: \mathcal{E}_1\to \mathcal{E}_2$ induces a cochain map 
$$
\underline{F}_X(\phi): (\mathfrak{E}^{\bullet}_1,d)\to (\mathfrak{E}^{\bullet}_2,d).
$$
 Moreover if $\phi$ is a homotopy equivalence, then $\underline{F}_X(\phi)$ is homotopic invertible. If $\phi$ is a gauge equivalence, then on each degree $k$, the map 
$$\underline{F}_X(\phi): \mathfrak{E}^k_1\to  \mathfrak{E}^k_2$$ is an isomorphism.
\end{prop}
\begin{proof}
It is a direct consequence of the definition.
\end{proof}

\begin{rmk}
In \cite[Theorem 8.3]{chuang2021maurer}, the result is stated for the derived category of globally bounded perfect complexes instead of $D^{\gb}_{\coh}(X)$. Nevertheless it is easy to see that these two categories are equivalent for nonsingular $X$.
\end{rmk}

For later applications we want to explicitly state the  following results, which are implied in Theorem \ref{thm: equiv of cats} and \ref{thm: equiv of cats noncompact}.

\begin{coro}\label{coro: quasi-isom is homotopy equivalence}
Any quasi-isomorphism in $D^{\gb}_{\coh}(X)$ is induced by a homotopy equivalence in $B(X)$.
\end{coro}

\begin{coro}\label{coro: Extension and morphism in B(X)}
For $\mathfrak{S}\in D^{\gb}_{\coh}(X)$ and  $\mathcal{E}\in B(X)$ such that $\underline{F}_X(\mathcal{E})\simeq \mathfrak{S}$, we have
\begin{equation}
\Hom_{D^{\gb}_{\coh}(X)}(\mathfrak{S},\mathfrak{S}[i])\cong \Hom_{\underline{B}(X)}(\mathcal{E},\mathcal{E}[i])\text{, for any } i.
\end{equation}
In particular if $\mathfrak{S}$ is a single globally bounded coherent sheaf, then
\begin{equation}
\Ext^i_X(\mathfrak{S},\mathfrak{S})\cong \Hom_{\underline{B}(X)}(\mathcal{E},\mathcal{E}[i])\text{, for any } i\geq 0.
\end{equation}
\end{coro}

Recall  that  we have the pull-back dg-functor $f^*_b: B(Y)\to B(X)$ and the induced functor $\underline{f^*_b}: \underline{B}(Y)\to \underline{B}(X)$. We have the following result.

\begin{prop}\label{prop: pull-back is the derived pull-back}
Under the equivalence of categories in Theorem \ref{thm: equiv of cats} and Theorem \ref{thm: equiv of cats noncompact}, $\underline{f^*_b}: \underline{B}(Y)\to \underline{B}(X)$ is compatible with the left derived pull-back functor $Lf^*: D^{\gb}_{\coh}(Y)\to D^{\gb}_{\coh}(X)$.
\end{prop}
\begin{proof}
The proof is the same as that of \cite[Proposition 6.6]{bismut2023coherent}: We can check that for any $\mathcal{E}\in \underline{B}(Y)$, its image $\underline{F}_Y(\mathcal{E})\in  D^{\gb}_{\coh}(Y)$ is a bounded complex of flat $\mathcal{O}_Y$-modules. Then the proposition is a consequence of Lemma \ref{lemma: derived pull-back flat modules} and Definition \ref{defi: pull-back of cohesive modules}. Notice that we do not need $X$ or $Y$ to be compact.
\end{proof}

\subsection{Currents and cohesive modules}\label{subsection: currents and cohesive modules}
For the definition  of currents on complex manifolds, see \cite[Chapter 3, Section 1]{griffiths1994principles}. Let $\mathcal{D}^{p,q}_X$ denote the sheaf of $(p,q)$-currents on $X$. There is a natural embedding $\Omega^{p,q}_X\hookrightarrow \mathcal{D}^{p,q}_X$.

Let $\mathcal{E}= (E^{\bullet}, A^{E^{\bullet}\prime\prime})$ be a cohesive module on $X$. Recall we have
$$
A^{E^{\bullet}\prime\prime}=v_0+\nabla^{E^{\bullet}\prime\prime}+v_2+\ldots
$$
It is clear that $\nabla^{E^{\bullet}\prime\prime}$ induces a map
\begin{equation}
\nabla^{E^{\bullet}\prime\prime}: \mathcal{D}^{p,q}_X\otimes E^{\bullet}\to \mathcal{D}^{p,q+1}_X\otimes E^{\bullet}
\end{equation}
and for $i\neq 1$, $v_i$ induces a map
\begin{equation}
\nabla^{E^{\bullet}\prime\prime}: \mathcal{D}^{p,q}_X\otimes E^{\bullet}\to \mathcal{D}^{p,q+1-i}_X\otimes E^{\bullet+i}.
\end{equation}
Similar to the construction in Section \ref{subsection: cohesive and coherent}, we can define a cochain complex of sheaves $\tilde{F}_X(\mathcal{E})=(\tilde{\mathfrak{E}}^{\bullet},d)$ where
\begin{equation}\label{eq: defi of tilde F}
\tilde{\mathfrak{E}}^n(U)=\bigoplus_{p+q=n}\Gamma(U, \mathcal{D}^{0,p}_X\otimes E^q)
\end{equation}
and $d: \tilde{\mathfrak{E}}^n\to \tilde{\mathfrak{E}}^{n+1}$ is exactly $A^{E^{\bullet}\prime\prime}$. It is clear that $(\tilde{\mathfrak{E}}^{\bullet},d)$ is also a cochain complex of sheaves of $\mathcal{O}_X$-modules. Moreover the embedding $\Omega^{p,q}_X\hookrightarrow \mathcal{D}^{p,q}_X$ induces a cochain map $i: (\mathfrak{E}^{\bullet},d)\to (\tilde{\mathfrak{E}}^{\bullet},d)$.

\begin{prop}\label{prop: form to current is a quasi-isom}
The above  cochain map $i: (\mathfrak{E}^{\bullet},d)\to (\tilde{\mathfrak{E}}^{\bullet},d)$ is a quasi-isomorphism of cochain complexes of sheaves of $\mathcal{O}_X$-modules.
\end{prop}
\begin{proof}
The claim is local so it is sufficient to prove the proposition on a small open subset $V\subset X$. By Proposition \ref{prop: gauge equivalence to flat}, for $V$ sufficiently small, we have a gauge equivalence $J: (E^{\bullet}, A^{E^{\bullet}\prime\prime})|_V\overset{\sim}{\to} (E^{\bullet}|_V, v_0+\nabla^{E^{\bullet}|_V\prime\prime})$, which induces  automorphisms
$$
J: \mathfrak{E}^n\overset{\sim}{\to} \mathfrak{E}^n
\text{ and } \tilde{J}: \tilde{\mathfrak{E}}^n|_V\overset{\sim}{\to} \tilde{\mathfrak{E}}^n|_V
$$
for each $n$. See Proposition \ref{prop: sheafify morphisms}. Let $\hat{d}$ denote the cochain map $\mathfrak{E}^n|_V\to \mathfrak{E}^{n+1}|_V$ and $\tilde{\mathfrak{E}}^n|_V\to \tilde{\mathfrak{E}}^{n+1}|_V$ induced by $v_0+\nabla^{E^{\bullet}|_V\prime\prime}$. We thus obtain  degreewise isomorphisms
\begin{equation}
J: (\mathfrak{E}^{\bullet}|_V,d)\to (\mathfrak{E}^{\bullet}|_V,\hat{d})
\text{ and }\tilde{J}: (\tilde{\mathfrak{E}}^{\bullet}|_V,d)\to (\tilde{\mathfrak{E}}^{\bullet}|_V,\hat{d})
\end{equation}
which are compatible with the embedding $i: \mathfrak{E}^{\bullet}|_V\to \tilde{\mathfrak{E}}^{\bullet}|_V$. Therefore it is sufficient to prove that
$$
i: (\mathfrak{E}^{\bullet}|_V,\hat{d})\to (\tilde{\mathfrak{E}}^{\bullet}|_V,\hat{d})
$$
is a quasi-isomorphism. Now $\nabla^{E^{\bullet}|_V\prime\prime}$ gives $E^{\bullet}|_V$ a structure of holomorphic vector bundle on $V$, so $(\mathfrak{E}^{\bullet}|_V,\hat{d})$ is the Dolbeault complex associated to a bounded cochain complex of holomorphic vector bundles. The claim is an easy consequence of standard results  in complex geometry as in \cite[Chapter 3, Section 1]{griffiths1994principles}.
\end{proof}

\begin{coro}\label{coro: global regularity section}
For $x\in \Gamma(X, \mathfrak{E}^n)$, if there exists $\tilde{y}\in \Gamma(X,  \tilde{\mathfrak{E}}^{n-1})$ such that $d(\tilde{y})=x$, then there exists $y\in \Gamma(X,\mathfrak{E}^{n-1})$ such that $d(y)=x$.
\end{coro}
\begin{proof}
It is a direct consequence of Proposition \ref{prop: form to current is a quasi-isom} and the fact that both $\mathfrak{E}^n$ and $ \tilde{\mathfrak{E}}^n$ are soft sheaves for each $n$.
\end{proof}

\section{Pseudomeromorphic and almost semimeromorphic currents}\label{section: pseudomeromorphic currents}
\label{section:pseudo}
In this section we review pseudomeromorphic and almost semimeromorphic currents following  \cite{andersson2010decomposition} and \cite{andersson2018direct}.

\subsection{Scalar valued currents}
Let $s$ be a holomorphic section of a Hermitian holomorphic line bundle $L$ over $X$.
The \emph{principal value current} $[1/s]$ can be defined as 
$$
[1/s] := \lim_{\epsilon \to 0} \chi(|s|^2/\epsilon)\frac{1}{s},
$$
where $\chi : \R \to \R$ is a smooth cut-off function, i.e., $\chi(t) = 0$ in a neighborhood of zero and $\chi(t) = 1$ when $|t| \gg 1$.
A current is \emph{semimeromorphic} if it is of the form $[\omega/s] := \omega[1/s]$, where $\omega$ is a smooth form with values in $L$.

Recall that a modification is a proper surjective holomorphic  map $\pi: X' \to X$ where $X$ and $X'$ are complex spaces, such that there exists a nowhere dense analytic subset $E\subset X$ such that
$$
\pi|_{X'\backslash \pi^{-1}(E)}: X'\backslash \pi^{-1}(E)\to X\backslash E
$$
is a biholomorphic isomorphism.

\begin{defi}\label{defi: almost semimeromorphic current}
A current $a$ is \emph{almost semimeromorphic} on $X$, written $a \in \text{ASM}(X)$, if there is a modification $\pi : X' \to X$ such that
$$a=\pi_*(\omega/s),$$ where $\omega/s$ is a semimeromorphic current in $X'$.

A current $a$ is \emph{locally} almost semimeromorphic on $X$, written $a \in \text{LASM}(X)$, if ther e is an open cover $\{U_i\}$ of $X$ such that $a|_{U_i}\in \text{ASM}(U_i)$ for each $U_i$.

For $a \in \text{LASM}(X)$, the \emph{Zariski-singular support} of $a$, denoted by ZSS$(a)$ is the smallest analytic subset of $X$ where $a$ is not smooth. ZSS$(a)$ has positive codimension in $X$.
\end{defi}

\begin{rmk}
ZSS$(a)$ is not the \emph{support} of $a$. The latter is defined for general currents.
\end{rmk}

\begin{prop}\label{prop: almost semimeromorphic currents form an algebra}
 (Locally) almost semimeromorphic currents on $X$ form a graded commutative algebra over
smooth forms. The class of (locally) almost semimeromorphic currents on $X$ is closed under $\partial$.
\end{prop}
\begin{proof}
For the almost semimeromorphic case see \cite[Section 4.1 and Proposition 4.16]{andersson2018direct}.  The locally almost semimeromorphic case follows immediately.
\end{proof}

In general LASM$(X)$ is not closed under $\dbar$. Actually we have the following more general concept: For an open subset $U\subset \C^N$ with coordinates $(t_1,\ldots, t_N)$, we have 
\begin{equation}\label{eq: elementary current}
\tau:=\dbar\Big[\frac{1}{t_{i_1}^{a_{i_1}}}\Big]\wedge \ldots \wedge \dbar\Big[\frac{1}{t_{i_q}^{a_{i_q}}}\Big]\wedge\Big[\frac{1}{t_{i_{q+1}}^{a_{i_{q+1}}}}\Big]\wedge\ldots \wedge\Big[\frac{1}{t_{i_{q+k}}^{a_{i_{q+k}}}}\Big]\wedge \alpha
\end{equation}
where $a_{i_1},\ldots, a_{i_{q+k}}\geq 1$ and $\alpha$ is a $C^{\infty}$-form on $U$ with compact support. According to \cite[Section 2]{andersson2018direct}, $\tau$ is a well-defined current. 
It $\tau$ is a current on a complex manifold $X$, we call $\tau$ an \emph{elementary current} if there exists a local chart $\{U_{\sigma}\}$ of $X$ such that $\tau$ is of the form of \eqref{eq: elementary current} when restricted to each $U_{\sigma}$. 

\begin{defi}[\cite{andersson2010decomposition} Section 2]\label{defi: pseudomeromorphic current}
Let $X$ be a complex manifold (or more generally, a complex analytic space). A current $T$ on $X$ is said to be a \emph{pseudomeromorphic current} if it can be written as a locally finite sum
\begin{equation}
T=\sum \Pi_* \tau_l
\end{equation}
where $\tau_l$ is an elementary current on some complex manifold $\tilde{X}_r$ and $\Pi=\Pi_1\circ \ldots \circ \Pi_r$ is a composition of resolutions of singularities
$$
\Pi_1: \tilde{X}_1\to X_1\subset X,\ldots, \Pi_r: \tilde{X}_r\to X_r\subset \tilde{X}_{r-1}.
$$
We denote the set of pseudomeromorphic currents on $X$ by PM$(X)$.
\end{defi}

Locally almost semimeromorphic currents are special cases of  pseudomeromorphic currents.

\begin{prop}\label{prop: pseudomeromorphic currents closed under multiplication}
The class of pseudomeromorphic currents is closed under multiplication with smooth forms and under $\partial$ and $\dbar$.  Moreover, a locally almost semimeromorphic current can act on a pseudomeromorphic current from both sides.
\end{prop}
\begin{proof}
See \cite[Section 2.1 and Section 4.2]{andersson2018direct}. Notice that although \cite[Section 4.2]{andersson2018direct} only discusses left action, we can define right action in the same way.
\end{proof}

Let $Z\subset X$ be an analytic subvariety. Integration along $Z$ gives a current on $X$ which we denote by $[Z]$. In particular if $Z$ has \emph{pure} codimension $p$ in $X$, i.e. every irreducible component of $Z$ has the same codimension $p$, then $[Z]$ is a $(p,p)$-current on $X$

\begin{prop}\label{prop: [Z] is pseudomeromorphic}
Let $Z\subset X$ be an analytic subvariety. Then the current $[Z]$ is a pseudomeromorphic current on $X$.
\end{prop}
\begin{proof}
It is actually implied by the local computation as in \cite[Theorem 1.1]{andersson2005residues}.
\end{proof}

One important property of pseudomeromorphic currents is that they satisfy the following \emph{dimension principle}.

\begin{prop} [\cite{andersson2010decomposition} Corollary 2.4]\label{prop:dimPrinciple}
Let $T$ be a pseudomeromorphic $(*,q)$-current on $X$ with support on a subvariety $Z$.
If $\codim Z \geq q+1$, then $T = 0$.
\end{prop}

Given a pseudomeromorphic current $T$ and an analytic subset $Z$, as in \cite[Section 2]{andersson2010decomposition}, the restriction of $T$ to $X\backslash Z$ has an extension to $X$ in the following way: Let $\chi$
be a cut-off function as above. For a  local chart $U$ of $X$,  let  $F$ be a section of a holomorphic Hermitian vector bundle such that $Z\cap U = \{ F = 0 \}$. We define 
\begin{equation}\label{eq: chi_epsilon}
\chi_{\epsilon}:=\chi(|F|^2/\epsilon)
\end{equation}
and then
\begin{equation}\label{eq: defi of 1X/Z}
(\mathbf{1}_{X \backslash Z} T)|_U := \lim_{\epsilon \to 0} \chi(|F|^2/\epsilon) T|_U.
\end{equation}
It is clear that 
\begin{equation}\label{eq: 1X/Z|X/Z}
(\mathbf{1}_{X \backslash Z} T)|_{X \backslash Z}=T|_{X \backslash Z}.
\end{equation}

 By \cite[Lemma 2.6]{andersson2018direct}, the $(\mathbf{1}_{X \backslash Z} T)|_U$'s glue together to a pseudomeromorphic current $\mathbf{1}_{X \backslash Z} T$ on $X$. It is clear that we have
\begin{equation}\label{eq: restriction commutes with smooth forms}
\mathbf{1}_{X \backslash Z} (\alpha\wedge T)=\alpha\wedge \mathbf{1}_{X \backslash Z} T
\end{equation}
for any $C^{\infty}$-form $\alpha$.

\begin{defi}\label{defi: SEP}
A pseudomeromorphic current $T$ on $X$ is said to have the \emph{standard extension property} (SEP) if $\mathbf{1}_{X \backslash Z} T = T$
for any analytic subset $Z$ of positive codimension.
\end{defi}

\begin{prop}\label{prop: LASM has SEP}
Any $a\in \text{LASM}(X)$ has SEP.
\end{prop}
\begin{proof}
It follows from Definition \ref{defi: almost semimeromorphic current}, Proposition \ref{prop:dimPrinciple}, and \eqref{eq: restriction commutes with smooth forms}.
\end{proof}

\begin{defi}\label{defi: LASM extension}
Let $Z\subset X$ be an analytic subset of codimension $\geq 1$.  For $\alpha$ a smooth form on $X \backslash Z$, we say $\alpha$ has a LASM extension to $X$, if there exists an $a\in \text{LASM}(X)$ such that $a|_{X\backslash Z}=\alpha$.
\end{defi}

\begin{lemma}\label{lemma: uniqueness of LASM extension}
Let $Z\subset X$ be an analytic subset of codimension $\geq 1$.  If $\alpha$ is a smooth form on $X \backslash Z$, and $\alpha$ has
an extension as a locally almost semimeromorphic current $a$ on $X$, then such extension is unique.
\end{lemma}
\begin{proof}
If $a$ and $b$ are two such extensions, then $a|_{X\backslash Z}=b|_{X\backslash Z}=\alpha$. Since $a$ and $b$ are both LASM hence both have SEP, we know 
$$
a=\mathbf{1}_{X \backslash Z}(a|_{X\backslash Z})=\mathbf{1}_{X \backslash Z}\alpha=\mathbf{1}_{X \backslash Z}(b|_{X\backslash Z})=b.
$$
\end{proof}

\begin{coro}\label{coro: LASM extension glue}
Let  $Z\subset X$ be an analytic subset of codimension $\geq 1$.  If $\alpha$ is a smooth form on $X \backslash Z$ such that $\alpha$ locally has LASM extension, i.e. there exists an open cover $\{U_i\}$ of $X$ such that $\alpha|_{(X \backslash Z)\cap U_i}$ has a LASM extension to $U_i$ for each $i$, then $\alpha$ has a LASM extension to $X$.
\end{coro}
\begin{proof}
Let $a_i$ be the LASM extension of $\alpha|_{(X \backslash Z)\cap U_i}$ to $U_i$. By Lemma \ref{lemma: uniqueness of LASM extension}, $a_i|_{U_i\cap U_j}=a_j|_{U_i\cap U_j}$. We can then glue $a_i$ to a current $a$ on $X$ by partition of unity. $a$ is clearly LASM.
\end{proof}

 In particular, if $\alpha$ is a smooth form on $X \backslash Z$, and $\alpha$ has
an extension as a locally almost semimeromorphic current $a$ on $X$, then the extension is given by
\begin{equation} \label{eq:asmExtension}
    a = \lim_{\epsilon \to 0} \chi_{\epsilon} \alpha.
\end{equation}
where $\chi_{\epsilon}$ is given as in \eqref{eq: chi_epsilon}.

\begin{prop}\label{prop: restrict-ext of dbar a}
Let $a\in \text{ASM}(X)$. Let $Z=\text{ZSS}(a)$ be   the smallest analytic subset of $X$ where $a$ is not smooth. Then $\mathbf{1}_{X \backslash Z} (\dbar a)\in \text{ASM}(X)$.

Moreover if $a\in \text{LASM}(X)$ and  $Z=\text{ZSS}(a)$. Then $\mathbf{1}_{X \backslash Z} (\dbar a)\in \text{LASM}(X)$.
\end{prop}
\begin{proof}
The almost semimeromorphic case is proved in \cite[Proposition 4.16]{andersson2018direct}. The locally almost semimeromorphic case follows immediately.
\end{proof}

\begin{defi}\label{defi: residue of a current}
Let $a$ be a locally almost semimeromorphic current on $X$. Let $Z=\text{ZSS}(a)$ be as before
The \emph{residue} $R(a)$ of $a$ is defined by
\begin{equation}
\label{eq:residue-def}
	R(a) := \dbar a - \mathbf{1}_{X \backslash Z}\dbar a.
\end{equation}
\end{defi}

Note that
\begin{equation} \label{eq:resSupport}
    \supp R(a) \subseteq Z.
\end{equation}
Since $a$ is locally almost semimeromorphic, and thus has the SEP, it follows
by \eqref{eq:asmExtension} that $R(a)$ is locally given by
\begin{equation}
\label{eq:residue}
    R(a)=\lim_{\epsilon \to 0}
    \left(\dbar(\chi_\epsilon a) - \chi_\epsilon \dbar a \right)
    = \lim_{\epsilon \to 0} \dbar\chi_\epsilon \wedge a.
\end{equation}
It follows directly from for example \eqref{eq:residue} that if $\psi$ is a smooth form, then
\begin{equation} \label{eq:residueSmooth}
    R(\psi \wedge a) = (-1)^{\deg \psi} \psi\wedge R(a).
\end{equation}

\subsection{Bundle valued currents}
Let $X$ be a complex manifold and $E$ be a $C^{\infty}$-complex vector bundle on $X$. We can define   almost semimeromorphic, locally almost semimeromorphic, and pseudomeromorphic currents on $X$ valued in $E$ in the same way and we denote them by ASM$(X,E)$, LASM$(X,E)$, and PM$(X,E)$, respectively. In the same way we can  define ASM$(X,\End(E))$, LASM$(X,\End(E))$, and PM$(X,\End(E))$.

All results and definitions except Proposition \ref{prop: restrict-ext of dbar a} and Definition \ref{defi: residue of a current} hold automatically in the bundle valued case.

\begin{prop}\label{prop: LASM extension bundle}
For any $a\in \text{LASM}(X,E)$ and  any $\dbar$-connection $\nabla''_E$   on $E$, let $Z=\text{ZSS}(a)$. Then $\mathbf{1}_{X \backslash Z}(\nabla''_E(a))\in \text{LASM}(X,E)$.
\end{prop}
\begin{proof}
The statement is local so we can assume that $\nabla''_E=\dbar+\omega$ where $\omega$ is a smooth $(0,1)$-form valued in $\End(E)$. We know $\mathbf{1}_{X \backslash Z}(\dbar(a))\in \text{LASM}(X,E)$ by Proposition \ref{prop: restrict-ext of dbar a}. Moreover $\omega a\in \text{LASM}(X,E)$ since $\text{LASM}(X,E)$ is an algebra over smooth forms. By Proposition \ref{prop: LASM has SEP}, $\omega a$ has SEP, hence $\mathbf{1}_{X \backslash Z}(\omega a)=\omega a\in \text{LASM}(X,E)$.

\begin{defi}\label{defi: residue of LASM}
For  $a\in \text{LASM}(X,E)$. Pick a $\dbar$-connection $\nabla''_E$   on $E$, we define the \emph{residue} $R(a)$ of $a$ as
\begin{equation}\label{eq: residue-def-bundle}
R(a)=\nabla''_E(a)-\mathbf{1}_{X \backslash Z}\nabla''_E(a).
\end{equation}
\end{defi}
\end{proof}
It is easy to see that $R(a)$ is independent of the choice of the $\dbar$-connection $\nabla''_E$.

\section{Residue currents of cohesive modules}\label{section: residue currents of cohesive modules}
\subsection{Minimal right inverses of maps between vector bundles}
\begin{defi}\label{defi: minimal right inverse}
Let $E$  and $F$ be two complex vector spaces with Hermitian metrics. Let $\phi: E\to F$ be a complex linear map.
The \emph{minimal right inverse} of $\phi$ is a  map $\sigma: F\to E$ which satisfies
\begin{enumerate}
\item $(\phi\sigma)|_{\im \phi}=\id_{\im \phi}$;
\item $\sigma|_{(\im \phi)^{\bot}}=0$;
\item $\im \sigma\bot \ker \phi$
\end{enumerate}
on each fiber. In other words, since $\phi$ induces a fiberwise isomorphism $(\ker \phi)^{\bot}\overset{\sim}{\to}\im \phi$, $\sigma$ is  defined to be $\phi^{-1}$ on $\im \phi$ and $0$ on $(\im \phi)^{\bot}$. 
\end{defi}

Let $X$ be a smooth manifold and $\phi: E\to F$ be a map between $C^{\infty}$ vector bundles with Hermitian metrics. It is clear that rank$\phi$ is a lower semicontinuous function on $X$. Let $Z\subset X$ be the subset consisting of $x\in X$ such that $\im\phi_x$ does not get its maximal rank. Then $X\backslash Z$ is a nonempty open subset of $X$.  Let $\sigma$ be the fiberwise minimal right inverse of $\phi$. Then it is clear that $\sigma$ is a $C^{\infty}$-map from $F$ to $E$ when restricted to $X\backslash Z$. 

\begin{eg}
Let $\underline{\mathbb{C}}^m$ be the $n$-dimensional trivial vector bundle on $X$ equipped with the standard Hermitian metric. A map $\phi: \underline{\mathbb{C}}^n\to \underline{\mathbb{C}}$ is given by 
$$
\phi=(f_1,\ldots, f_m)
$$
where $f_1, \ldots, f_m$ are $C^{\infty}$-functions on $X$.

We need to distinguish two cases.
\begin{enumerate}
\item If all $f_i$'s are identically $0$ on $X$, then the maximal rank of $\im\phi$ is $0$, hence $Z=\emptyset$ and $\sigma\equiv 0$.

\item If some $f_i$'s are not identically $0$ on $X$, then  the maximal rank of $\im\phi$ is $1$, hence 
$$
Z=\{x\in X|f_1(x)=\ldots =f_m(x)=0\}
$$
and
$$
\sigma(x) = \begin{cases}
  0 & x\in Z \\
  \frac{1}{\sum_{i=1}^m|f_i|^2}\begin{pmatrix}\overline{f_1}\\ \ldots\\ \overline{f_m}\end{pmatrix} & x\in X\backslash Z.
\end{cases}
$$
\end{enumerate}

It is clear that in the second case, $\sigma(x)$ is $C^{\infty}$ on $X\backslash Z$ but not $C^{\infty}$ on $X$. Moreover, even if $X$ is a complex manifold and all $f_i$'s are holomorphic functions, $\sigma$ is not holomorphic even when restricted to  $X\backslash Z$.
\end{eg}

\subsection{Minimal right inverses and cohesive modules}

Now let $X$ be a complex manifold and $\mathcal{E}= (E^{\bullet}, A^{E^{\bullet}\prime\prime})$ be a cohesive module on $X$ as in Definition \ref{defi: cohesive module}, where
$$
A^{E^{\bullet}\prime\prime}=v_0+\nabla^{E^{\bullet}\prime\prime}+v_2+\ldots
$$
as before. Let $Z_i\subset X$ be the subset of $X$ consisting of $x\in X$ such that $v_0^i: E^i\to E^{i+1}$ does not get its maximal rank. 

\begin{prop}\label{prop: not maximal rank analytic subvariety}
Each $Z_i$ is an analytic subvariety of $X$ with codimension $\geq 1$.
\end{prop}
\begin{proof}
The claim is local. By \cite[Theorem 5.2.1]{bismut2023coherent}, for any $x\in X$, there exists a open neighborhood $U$ of $x$, on which we have a flat $\dbar$ connection $\dbar^{E^i}$ on each $E^i$ such that $\dbar^{E^{\bullet}}v_0=0$, i.e. $v_0^i: E^i\to E^{i+1}$ is a holomorphic map under this new holomorphic structure on $E^{\bullet}$. The claim then follows immediately.
\end{proof}

Let $Z:=\cup_i Z_i$. Then $Z$ is still an analytic subvariety of $X$ with codimension $\geq 1$.

We equip each $E^i$ with a Hermitian metric and call such $\mathcal{E}= (E^{\bullet}, A^{E^{\bullet}\prime\prime})$ a \emph{Hermitian} cohesive module. We do not assume any compatibility between $v_0$ and the metric.

Let $\sigma^i: E^{i+1}\to E^i$ be the fiberwise minimal right inverse of $v_0^i: E^i\to E^{i+1}$. Then $\sigma^i$ is a $C^{\infty}$-map   when restricted to $X\backslash Z$. To simplify the notation, we denote
\begin{equation}\label{eq: sigma is the sum of sigma i}
\sigma:=\sum_i \sigma^i\in \End^{-1}(E^{\bullet}).
\end{equation}

\begin{lemma}\label{lemma: sigma square=0}
We have $\sigma^2=0$.
\end{lemma}
\begin{proof}
By Definition \ref{defi: minimal right inverse}, we have
 $$\im \sigma^i=(\ker v_0^i)^{\bot}\subset (\im v_0^{i-1})^{\bot}
$$ and $\sigma^{i-1}|_{(\im v_0^{i-1})^{\bot}}=0$. Hence $\sigma^{i-1}\sigma^i=0$ for each $i$.
\end{proof}

\subsection{The residue current of a cohesive module}\label{subsection: residue current of a cohesive module definition}
 Let $X$ be a complex manifold and $\mathcal{E}= (E^{\bullet}, A^{E^{\bullet}\prime\prime})$ be a Hermitian cohesive module on $X$. Let $Z=\cup_i Z_i$ be as before and $\sigma$ be as in \eqref{eq: sigma is the sum of sigma i}. We denote
$$
A^{E^{\bullet}\prime\prime}_{\geq 1}:= \nabla^{E^{\bullet}\prime\prime}+v_2+\ldots=A^{E^{\bullet}\prime\prime}-v_0.
$$
The $\dbar$-connection $ \nabla^{E^{\bullet}\prime\prime}$ induces a $\dbar$-connection on $\End{E^{\bullet}}$, which we still denote by $\nabla^{E^{\bullet}\prime\prime}$. 

Let $V\subset X$ be an open subset. For $a\in \Gamma(V,\Omega^{\bullet, \bullet}_X\hat{\otimes} \End(E^{\bullet}))$, we define $A^{E^{\bullet}\prime\prime}_{\geq 1}(a)\in \Gamma(V,\Omega^{\bullet, \bullet}_X\hat{\otimes} \End(E^{\bullet}))$ as
\begin{equation}\label{eq: definition of A geq 1}
A^{E^{\bullet}\prime\prime}_{\geq 1}(a):=\nabla^{E^{\bullet}\prime\prime}(a)+[v_2, a]+[v_3,a]+\ldots
\end{equation}
where $[v_i,a]$ is the graded commutator with respect to the total degree. We can define $A^{E^{\bullet}\prime\prime}(a)$ in a similar way.

We know that $\sigma\in C^{\infty}(X\backslash Z, \End^{-1}(E^{\bullet}))$ when we restrict it to $X\backslash Z$. We then define 
$$
u^{\mathcal{E}}\in \Gamma(X\backslash Z, \Omega^{0, \bullet}_X\hat{\otimes} \End(E^{\bullet}))
$$ of total degree $-1$ as 
\begin{equation}\label{eq: definition of u}
u^{\mathcal{E}}:=\sigma(\id_{E^{\bullet}}+ A^{E^{\bullet}\prime\prime}_{\geq 1}(\sigma))^{-1}=\sigma-\sigma A^{E^{\bullet}\prime\prime}_{\geq 1}(\sigma)+\sigma (A^{E^{\bullet}\prime\prime}_{\geq 1}(\sigma))^2-\ldots
\end{equation}

\begin{rmk}
Since $A^{E^{\bullet}\prime\prime}_{\geq 1}(\sigma)$ is in $\Gamma(X\backslash Z,  \Omega^{0,\geq 1}_X\hat{\otimes} \End(E^{\bullet}))$, the sum on the right hand side of \eqref{eq: definition of u} is finite.
\end{rmk}

\begin{rmk}
In \cite[Equation (4.2)]{johansson2021explicit}, the analogue of $u^{\mathcal{E}}$ for twisting cochains is given by
$$
u=\sigma(\id-\dbar\sigma)^{-1}.
$$
In a private communication, L\"{a}rk\"{a}ng showed the author that the $u$ in \cite{johansson2021explicit} is actually equal to
$$
\sigma^0\big(\id+(a'(\sigma^0)-\dbar(\sigma^0))\big)^{-1}
$$
which is analogous to the $u^{\mathcal{E}}$  in \eqref{eq: definition of u}.
\end{rmk}

For later applications we need the following lemma.

\begin{lemma}\label{lemma: A sigma sigma=sigma A sigma}
For any $j\geq 0$, we have 
\begin{equation}\label{eq: A sigma sigma=sigma A sigma}
\sigma (A^{E^{\bullet}\prime\prime}_{\geq 1}(\sigma))^j= (A^{E^{\bullet}\prime\prime}_{\geq 1}(\sigma))^j\sigma
\end{equation}
\end{lemma}
\begin{proof}
By Lemma \ref{lemma: sigma square=0} we have $\sigma\sigma=0$. Since $A^{E^{\bullet}\prime\prime}_{\geq 1}$ is a derivation and $\sigma$ has degree $-1$, we have
\begin{equation}
A^{E^{\bullet}\prime\prime}_{\geq 1}(\sigma)\sigma=\sigma A^{E^{\bullet}\prime\prime}_{\geq 1}(\sigma).
\end{equation}
\eqref{eq: A sigma sigma=sigma A sigma} then follows immediately.
\end{proof}

\begin{prop}\label{prop: almost semimeromorphic extension of u}
The form $u^{\mathcal{E}}$ has a locally almost semimeromorphic (LASM) extension to $X$.
\end{prop}
\begin{proof}
By the same argument as in \cite[Example 4.18]{andersson2018direct} we know that $\sigma$ has an extension to a LASM current on $X$. We then consider 
$$
A^{E^{\bullet}\prime\prime}_{\geq 1}(\sigma):=\nabla^{E^{\bullet}\prime\prime}(\sigma)+[v_2, \sigma]+[v_3,\sigma]+\ldots
$$

Since LASM currents form an algebra over smooth forms,  $[v_i, \sigma]$ has an extension to a LASM current on $X$ for $i\geq 2$. By Proposition \ref{prop: LASM extension bundle},  $\nabla^{E^{\bullet}\prime\prime}(\sigma)$ also has an extension to a LASM current on $X$. Hence $A^{E^{\bullet}\prime\prime}_{\geq 1}(\sigma)$ has an extension to a LASM current on $X$.

Finally by \eqref{eq: definition of u}, $u^{\mathcal{E}} =\sigma-\sigma A^{E^{\bullet}\prime\prime}_{\geq 1}(\sigma)+\sigma (A^{E^{\bullet}\prime\prime}_{\geq 1}(\sigma))^2-\ldots$ also has an extension to a LASM current on $X$. 
\end{proof} 

Let $U^{\mathcal{E}}$ be the LASM extension of $u^{\mathcal{E}}$ to $X$. By \eqref{eq:asmExtension}, 
\begin{equation}\label{eq: def of U}
U^{\mathcal{E}}=\lim_{\epsilon\to 0} \chi_{\epsilon}u^{\mathcal{E}}.
\end{equation}
$U^{\mathcal{E}}$ is an $\End(E^{\bullet})$-valued  LASM $(0,\bullet)$-current on $X$ with total degree $-1$.

\begin{defi}\label{defi: residue current of E}
 Let $X$ be a complex manifold and $\mathcal{E}= (E^{\bullet}, A^{E^{\bullet}\prime\prime})$ be a Hermitian cohesive module on $X$. Let $U^{\mathcal{E}}$ be as above, We define the \emph{residue current} $R^{\mathcal{E}}$ associated to $\mathcal{E}$ as
\begin{equation}\label{eq: RE}
R^{\mathcal{E}}:=\id_{E^{\bullet}}- A^{E^{\bullet}\prime\prime}(U^{\mathcal{E}})=\id_{E^{\bullet}}- A^{E^{\bullet}\prime\prime}U^{\mathcal{E}}- U^{\mathcal{E}} A^{E^{\bullet}\prime\prime}.
\end{equation}
$R^{\mathcal{E}}$ is an $\End(E^{\bullet})$-valued pseudomeromorphic (PM)  $(0,\bullet)$-current on $X$ with total degree $0$.
\end{defi}

It is clear that $R^{\mathcal{E}}$ satisfies
\begin{equation}
A^{E^{\bullet}\prime\prime}(R^{\mathcal{E}})=0.
\end{equation}

\begin{rmk}
If $\mathcal{E}= (E^{\bullet}, A^{E^{\bullet}\prime\prime})$ is a bounded complex of Hermitian holomorphic vector bundles, i.e. $v_i=0$ for $i\geq 2$, then $R^{\mathcal{E}}$ coincide with the residue current constructed in \cite[Section 2]{andersson2007residue}.
\end{rmk}

\begin{rmk}
In general $R^{\mathcal{E}}$ is not a  LASM current.
\end{rmk}

\begin{defi}\label{defi: R(U) and tilde R}
Let $X$, $Z$, and $U^{\mathcal{E}}$ be as above. Recall the \emph{residue} $R(U^{\mathcal{E}})$ of $U^{\mathcal{E}}$ is the current
\begin{equation}
R(U^{\mathcal{E}}):=\nabla^{E^{\bullet}\prime\prime}(U^{\mathcal{E}})-\mathbf{1}_{X \backslash Z}\nabla^{E^{\bullet}\prime\prime}(U^{\mathcal{E}})
\end{equation}
We define the current $\tilde{R}^{\mathcal{E}}$ as 
\begin{equation}\label{eq: tilde R and R}
\tilde{R}^{\mathcal{E}}:=R^{\mathcal{E}}+R(U^{\mathcal{E}}).
\end{equation}
\end{defi}

\begin{lemma}\label{lemma: tilde R further}
 We have 
\begin{equation}\label{eq: tilde R and R further}
\tilde{R}^{\mathcal{E}}=\id_{E^{\bullet}}-\mathbf{1}_{X \backslash Z}A^{E^{\bullet}\prime\prime}(U^{\mathcal{E}}).
\end{equation}
\end{lemma}
\begin{proof}
Since $U^{\mathcal{E}}$ is a  LASM current and $v_i$ is smooth for each $i\neq 1$, we know that $[v_i, U^{\mathcal{E}}]$ is   LASM for each $i\neq 1$. Hence 
\begin{equation}
R(U^{\mathcal{E}})=A^{E^{\bullet}\prime\prime}(U^{\mathcal{E}})-\mathbf{1}_{X \backslash Z}A^{E^{\bullet}\prime\prime}(U^{\mathcal{E}})
\end{equation}
and \eqref{eq: tilde R and R further} follows.
\end{proof}

It is clear that $R(U^{\mathcal{E}})|_{X \backslash Z}=0$ hence
\begin{equation}\label{eq: R and R tilde equal}
\tilde{R}^{\mathcal{E}}|_{X \backslash Z}=R^{\mathcal{E}}|_{X \backslash Z}.
\end{equation}

\begin{lemma}\label{lemma: tilde R is LASM}
The current $\tilde{R}^{\mathcal{E}}$ in Definition \ref{defi: R(U) and tilde R} is a LASM current. Moreover it is the (unique) LASM extension of $\id_{E^{\bullet}}-A^{E^{\bullet}\prime\prime}(u^{\mathcal{E}})$ to $X$.
\end{lemma}
\begin{proof}
Since both $\id_{E^{\bullet}}$ and $\mathbf{1}_{X \backslash Z}A^{E^{\bullet}\prime\prime}(U^{\mathcal{E}})$ are LASM currents, it is clear that $\tilde{R}^{\mathcal{E}}$ is  LASM.
\end{proof}

\begin{rmk}
Conceptually \eqref{eq: tilde R and R} means that $R^{\mathcal{E}}$ can be decomposed into the difference of the LASM part  $\tilde{R}^{\mathcal{E}}$ and the residual part $R(U^{\mathcal{E}})$.
\end{rmk}

We will use the following notation frequently in this paper.

\begin{defi}\label{defi: (l,k)-component residue current}
We denote by $R^{\mathcal{E}}_{q\to l}$  the component of $R^{\mathcal{E}}$ that maps
$\Gamma(X, \Omega^{0,\bullet}\hat{\otimes} E^q)$ to $\Gamma(X,\mathcal{D}^{0,\bullet}\hat{\otimes} E^l)$. We use similar notations for $R(U^{\mathcal{E}})$ and $\tilde{R}^{\mathcal{E}}$.
\end{defi}

Recall we define the complex of sheaves $\underline{F}_X(\mathcal{E})=(\mathfrak{E}^{\bullet},d)$ as in \eqref{eq: sheafify of cohesive module}.
We have the following result, which  generalizes the \emph{duality principle} in \cite[Proposition 2.3]{andersson2007residue}.

\begin{thm}\label{thm: duality principle}
Let $X$ be a complex manifold and $\mathcal{E}= (E^{\bullet}, A^{E^{\bullet}\prime\prime})$ be a Hermitian cohesive module on $X$. Let 
$$
s\in \Gamma(X, \mathfrak{E}^k)=\bigoplus_{p+q=k}\Gamma(X, \Omega^{0,p}\hat{\otimes} E^q)
$$
be such that $A^{E^{\bullet}\prime\prime}(s)=0$. 
\begin{enumerate}
\item If  $R^{\mathcal{E}}(s)=0$,  then there exists a 
$$
t\in  \Gamma(X, \mathfrak{E}^{k-1})= \bigoplus_{p+q=k-1}\Gamma(X, \Omega^{0,p}\hat{\otimes} E^q)
$$ such that $A^{E^{\bullet}\prime\prime}(t)=s$.
\item If $R^{\mathcal{E}}_{q\to l}=0$ for any $q\leq k-1$ and any $l$. If there exists a 
$$
t\in  \Gamma(X, \mathfrak{E}^{k-1})=\bigoplus_{p+q=k-1}\Gamma(X, \Omega^{0,p}\hat{\otimes} E^q)
$$ such that $A^{E^{\bullet}\prime\prime}(t)=s$, then $R^{\mathcal{E}}(s)=0$.
\end{enumerate}
\end{thm}
\begin{proof}
If $R^{\mathcal{E}}(s)=0$, then by \eqref{eq: RE} we have
$$
0=s-A^{E^{\bullet}\prime\prime}(U^{\mathcal{E}} (s))- U^{\mathcal{E}} (A^{E^{\bullet}\prime\prime}(s)).
$$
Since $A^{E^{\bullet}\prime\prime}(s)=0$, we get
$$
s=A^{E^{\bullet}\prime\prime}(U^{\mathcal{E}} (s))
$$
with $U^{\mathcal{E}}(s)\in \tilde{\mathfrak{E}}^{k-1}$ where $\tilde{\mathfrak{E}}^{k-1}= \bigoplus_{p+q=k-1}\Gamma(X,  \mathcal{D}_X^{0,p}\otimes E^q)$ as in \eqref{eq: defi of tilde F}. By Corollary \ref{coro: global regularity section}, there exists a 
$$
t\in \Gamma(X, \mathfrak{E}^{k-1})=\bigoplus_{p+q=k-1}\Gamma(X, \Omega^{0,p}\hat{\otimes} E^q)
$$
such that $A^{E^{\bullet}\prime\prime}(t)=s$.

On the other hand if $A^{E^{\bullet}\prime\prime}(t)=s$. Since $A^{E^{\bullet}\prime\prime}(R^{\mathcal{E}})=0$ we get
$$
R^{\mathcal{E}}(s)=A^{E^{\bullet}\prime\prime}(R^{\mathcal{E}}(t)).
$$
Since $t\in \bigoplus_{p+q=k-1}\Gamma(X, \Omega^{0,p}\hat{\otimes} E^q)$ we get $R^{\mathcal{E}}(t)=0$ hence  $R^{\mathcal{E}}(s)=0$.
\end{proof}

\begin{rmk}
\cite[Proposition 4.2]{Johansson2023residue} gives a similar result in the framework of twisting cochains.
\end{rmk}

\begin{rmk}
We will see in Section \ref{section: vanishing of residue current} cases that $R^{\mathcal{E}}_{q\to l}$ indeed vanishes for any $q\leq k-1$ and any $l$.
\end{rmk}

\section{Vanishing of residue currents}\label{section: vanishing of residue current}

Let $\mathcal{E}= (E^{\bullet}, A^{E^{\bullet}\prime\prime})$ be a Hermitian cohesive module on a complex manifold $X$. By \eqref{eq: tilde R and R} we can decompose 
the residue $R^{\mathcal{E}}$ as
\begin{equation}
R^{\mathcal{E}}=\tilde{R}^{\mathcal{E}}-R(U^{\mathcal{E}}).
\end{equation}
We will study the vanishing of  $R(U^{\mathcal{E}})$ and $\tilde{R}^{\mathcal{E}}$ separately.

\subsection{Vanishing of $R(U^{\mathcal{E}})$}

We first study the vanishing conditions of $R(U^{\mathcal{E}})$.

 Recall that  $Z_i\subset X$ is the subvariety of $X$ consisting of $x\in X$ such that $v_0^i: E^i\to E^{i+1}$ does not get its maximal rank.

We have the following vanishing result on $R( U^{\mathcal{E}})$, which is an analogue to \cite[Proposition 4.4]{Johansson2023residue}.

\begin{prop}\label{prop: vanishing of R(U)}
Let $U^{\mathcal{E}}$ be the current defined in \eqref{eq: def of U} and $R(U^{\mathcal{E}})$ be its residue as in  Definition \ref{defi: R(U) and tilde R}. Then for any $k\geq q$ we have
\begin{equation}\label{eq: vanishing of R(U) trivial}
R(U^{\mathcal{E}})_{q\to k}=0.
\end{equation}

Moreover if there exists a pair of integers $l$, $q$ such that $l \leq q-1$ and the subvarieties $Z_i$'s satisfy
\begin{equation}\label{eq: codim bound of Zm}
\codim(Z_m)\geq q-m+1 \text{, for } l\leq m\leq  q-1,
\end{equation}
then for any $k\geq l$ we have
\begin{equation}\label{eq: vanishing of R(U)}
R(U^{\mathcal{E}})_{q \to k}=0.
\end{equation}
where $R(U^{\mathcal{E}})_{q \to k}$ is the component of $R(U^{\mathcal{E}})$ as in Definition \ref{defi: (l,k)-component residue current}.
\end{prop}
\begin{rmk}
If $Z_m=\emptyset$, then we set $\codim(Z_m)=\infty$. 
\end{rmk}
\begin{proof}
Recall that $U^{\mathcal{E}}$ is the LASM extension of 
$$
u^{\mathcal{E}}=\sum_{j\geq 0}(-1)^j\sigma (A^{E^{\bullet}\prime\prime}_{\geq 1}(\sigma))^j
$$
Lemma \ref{lemma: A sigma sigma=sigma A sigma} tells us $\sigma (A^{E^{\bullet}\prime\prime}_{\geq 1}(\sigma))^j= (A^{E^{\bullet}\prime\prime}_{\geq 1}(\sigma))^j\sigma$.
To abuse the notation we denote the LASM extension of  $ (A^{E^{\bullet}\prime\prime}_{\geq 1}(\sigma))^j\sigma$ also by $ (A^{E^{\bullet}\prime\prime}_{\geq 1}(\sigma))^j\sigma$. We then have the residues $R( (A^{E^{\bullet}\prime\prime}_{\geq 1}(\sigma))^j\sigma)$. 

To prove \eqref{eq: vanishing of R(U) trivial} and \eqref{eq: vanishing of R(U)} it is sufficient to prove
\begin{equation}\label{eq: vanishing of R(sigma)}
R( (A^{E^{\bullet}\prime\prime}_{\geq 1}(\sigma))^j \sigma)_{q \to k}=0  \text{, for any }  j.
\end{equation}
Notice that since $\sigma$ lowers the $E^{\bullet}$ degree by $1$ and $A^{E^{\bullet}\prime\prime}_{\geq 1}(\sigma)$ lowers the $E^{\bullet}$ degree by at least $1$, \eqref{eq: vanishing of R(sigma)} holds for $j\geq q-k$ by degree reason. In particular \eqref{eq: vanishing of R(U) trivial} is trivial by degree reason.

We then prove the following lemma.

\begin{lemma}\label{lemma: support contained in Z}
We have the inclusion
\begin{equation}\label{eq: inclusion of support}
\supp[R( (A^{E^{\bullet}\prime\prime}_{\geq 1}(\sigma))^j \sigma)_{q \to k}]\subset \bigcup_{k\leq m\leq q-1} Z_m.
\end{equation}
\end{lemma}
\begin{proof}[Proof of Lemma \ref{lemma: support contained in Z}]
Recall that $\sigma$ is the sum of its component $\sigma^i: E^{i+1}\to E^i$ where the latter is the fiberwise minimal right inverse of $v_0^i: E^i\to E^{i+1}$. 

By definition 
$$
R( (A^{E^{\bullet}\prime\prime}_{\geq 1}(\sigma))^j \sigma)=\nabla^{E^{\bullet}\prime\prime}( A^{E^{\bullet}\prime\prime}_{\geq 1}(\sigma))^j \sigma)-\mathbf{1}_{X \backslash Z}\nabla^{E^{\bullet}\prime\prime}(A^{E^{\bullet}\prime\prime}_{\geq 1}(\sigma))^j \sigma).
$$
Since $\nabla^{E^{\bullet}\prime\prime}$ preserves the $E^{\bullet}$ degree, we have
\begin{equation}
R( (A^{E^{\bullet}\prime\prime}_{\geq 1}(\sigma))^j \sigma)_{q \to k}=R(( (A^{E^{\bullet}\prime\prime}_{\geq 1}(\sigma))^j \sigma)_{q \to k}).
\end{equation}
Recall 
$$
A^{E^{\bullet}\prime\prime}_{\geq 1}(\sigma)=\nabla^{E^{\bullet}\prime\prime}(\sigma)+[v_2, \sigma]+[v_3,\sigma]+\ldots
$$
We know $\nabla^{E^{\bullet}\prime\prime}$ preserves the $E^{\bullet}$ degree and the $v_i$'s lower the $E^{\bullet}$ degree. So the component $[ (A^{E^{\bullet}\prime\prime}_{\geq 1}(\sigma))^j \sigma]_{q \to k}$ only involves the $\sigma^m$'s with $k\leq m\leq q-1$.

 It is clear that $\sigma^m$ is smooth outside $Z_m$. So $[ (A^{E^{\bullet}\prime\prime}_{\geq 1}(\sigma))^j \sigma]_{q \to k}$ is smooth outside $\bigcup_{k\leq m\leq q-1} Z_m$. On the other hand the residue $R(a)$ vanishes on the open subset where $a$ is smooth. We then get \eqref{eq: inclusion of support}.
\end{proof}
We then proof Proposition \ref{prop: vanishing of R(U)} by downward induction on $k$. First for $k=q-1$, we only need to prove \eqref{eq: vanishing of R(sigma)} for $j$=0. Actually Lemma \ref{lemma: support contained in Z} and \eqref{eq: codim bound of Zm} tell us that $R( \sigma)_{q \to q-1}$ has support of codimension $\geq 2$. On the other hand  we know that $R( \sigma)_{q \to q-1}$ is a $(0,1)$-pseudomeromorphic (PM) current. So the dimension principle in Proposition \ref{prop:dimPrinciple} tells us that 
\begin{equation}
R( \sigma)_{q \to q-1}=0.
\end{equation}

Now consider $k_0\leq q-2$. Assume that \eqref{eq: vanishing of R(sigma)} holds for  $k=k_0+1,\ldots q-1$ and $j=0,\ldots q-k_0-2$. Consider $R( (A^{E^{\bullet}\prime\prime}_{\geq 1}(\sigma))^j \sigma)_{q \to k_0}$ for $j\geq 1$. We have 
\begin{equation}\label{eq: decompose of R into degree}
\begin{split}
&R( (A^{E^{\bullet}\prime\prime}_{\geq 1}(\sigma))^j \sigma)_{q \to k_0}\\
=&R(((A^{E^{\bullet}\prime\prime}_{\geq 1}(\sigma))^j \sigma)_{q \to k_0})\\
=&R((\nabla^{E^{\bullet}\prime\prime}(\sigma))_{k_0+1 \to k_0}((A^{E^{\bullet}\prime\prime}_{\geq 1}(\sigma))^{j-1} \sigma)_{q \to k_0+1})\\
+&\sum_{i=2}^{q-k_0-1}R(([v_i, \sigma](A^{E^{\bullet}\prime\prime}_{\geq 1}(\sigma))^{j-1} \sigma)_{q \to k_0}).
\end{split}
\end{equation}
As before we see that $(\nabla^{E^{\bullet}\prime\prime}(\sigma))_{k_0+1 \to k_0}$ is smooth outside $ Z_{k_0}$. By  \eqref{eq:residueSmooth}, outside $ Z_{k_0}$ we have 
\begin{equation}
\begin{split}
&R((\nabla^{E^{\bullet}\prime\prime}(\sigma))_{k_0+1 \to k_0}((A^{E^{\bullet}\prime\prime}_{\geq 1}(\sigma))^{j-1} \sigma)_{q \to k_0+1})\\
=&(\nabla^{E^{\bullet}\prime\prime}(\sigma))_{k_0+1 \to k_0}R(((A^{E^{\bullet}\prime\prime}_{\geq 1}(\sigma))^{j-1} \sigma)_{q \to k_0+1})
\end{split}
\end{equation}
which vanishes by the induction hypothesis. As a result we know the support of 
$$
R((\nabla^{E^{\bullet}\prime\prime}(\sigma))_{k_0+1 \to k_0}((A^{E^{\bullet}\prime\prime}_{\geq 1}(\sigma))^{j-1} \sigma)_{q \to k_0+1})
$$
is contained in $ Z_{k_0}$, whose codimension is at least $q-k_0+1$ by \eqref{eq: codim bound of Zm}. On the other hand
$$
R((\nabla^{E^{\bullet}\prime\prime}(\sigma))_{k_0+1 \to k_0}((A^{E^{\bullet}\prime\prime}_{\geq 1}(\sigma))^{j-1} \sigma)_{q \to k_0+1})
$$
is a $(0, q-k_0)$-PM current. So the dimension principle in Proposition \ref{prop:dimPrinciple} tells us that 
\begin{equation}
R((\nabla^{E^{\bullet}\prime\prime}(\sigma))_{k_0+1 \to k_0}((A^{E^{\bullet}\prime\prime}_{\geq 1}(\sigma))^{j-1} \sigma)_{q \to k_0+1})=0.
\end{equation}

Now for each $2\leq i\leq q-k_0-1$ we look at $R(([v_i, \sigma](A^{E^{\bullet}\prime\prime}_{\geq 1}(\sigma))^{j-1} \sigma)_{q \to k_0})$. We know that
\begin{equation}
\begin{split}
&R(([v_i, \sigma](A^{E^{\bullet}\prime\prime}_{\geq 1}(\sigma))^{j-1} \sigma)_{q \to k_0})\\
=&R((v_i\sigma(A^{E^{\bullet}\prime\prime}_{\geq 1}(\sigma))^{j-1} \sigma)_{q \to k_0})+R((\sigma v_i(A^{E^{\bullet}\prime\prime}_{\geq 1}(\sigma))^{j-1} \sigma)_{q \to k_0})
\end{split}
\end{equation}
Actually $R(v_i\sigma(A^{E^{\bullet}\prime\prime}_{\geq 1}(\sigma))^{j-1} \sigma)$ vanishes by Lemma \ref{lemma: sigma square=0} and Lemma \ref{lemma: A sigma sigma=sigma A sigma}.  On the other hand, we know that
\begin{equation}
\begin{split}
&R((\sigma v_i(A^{E^{\bullet}\prime\prime}_{\geq 1}(\sigma))^{j-1} \sigma)_{q \to k_0})\\
=&R(\sigma_{k_0+1 \to k_0} (v_i(A^{E^{\bullet}\prime\prime}_{\geq 1}(\sigma))^{j-1} \sigma)_{q \to k_0+1}).
\end{split}
\end{equation}
Again $\sigma_{k_0+1 \to k_0}$ is smooth  outside $ Z_{k_0}$. By  \eqref{eq:residueSmooth}, outside $ Z_{k_0}$ we have
\begin{equation}
\begin{split}
&R(\sigma_{k_0+1 \to k_0} (v_i(A^{E^{\bullet}\prime\prime}_{\geq 1}(\sigma))^{j-1} \sigma)_{q \to k_0+1})\\
=&(\sigma v_i)_{k_0+i \to k_0}R( ((A^{E^{\bullet}\prime\prime}_{\geq 1}(\sigma))^{j-1} \sigma)_{q \to k_0+i})
\end{split}
\end{equation}
which vanishes by the induction hypothesis. As a result we know the support of 
$$
R((\sigma v_i(A^{E^{\bullet}\prime\prime}_{\geq 1}(\sigma))^{j-1} \sigma)_{q \to k_0})
$$
is contained in $ Z_{k_0}$, whose codimension is at least $q-k_0+1$ by \eqref{eq: codim bound of Zm}. Again
$$
R((\sigma v_i(A^{E^{\bullet}\prime\prime}_{\geq 1}(\sigma))^{j-1} \sigma)_{q \to k_0})
$$
is a $(0, q-k_0)$-PM current. So the dimension principle Proposition \ref{prop:dimPrinciple} tells us that 
\begin{equation}
R((\sigma v_i(A^{E^{\bullet}\prime\prime}_{\geq 1}(\sigma))^{j-1} \sigma)_{q \to k_0})=0
\end{equation}
hence 
\begin{equation}
R(( [v_i,\sigma](A^{E^{\bullet}\prime\prime}_{\geq 1}(\sigma))^{j-1} \sigma)_{q \to k_0})=0
\end{equation}
for each $2\leq i\leq q-k_0-1$.  By \eqref{eq: decompose of R into degree} we get
\begin{equation}
R( (A^{E^{\bullet}\prime\prime}_{\geq 1}(\sigma))^j \sigma)_{q \to k_0}=0.
\end{equation}
We finished the induction hence completed the proof of Proposition \ref{prop: vanishing of R(U)}.
\end{proof}

For a Hermitian cohesive module $\mathcal{E}= (E^{\bullet}, A^{E^{\bullet}\prime\prime})$, recall that in Section \ref{subsection: cohesive and coherent} we defined the  functor $\underline{F}_X(\mathcal{E})=(\mathfrak{E}^{\bullet},d)$, which is a bounded complex with (globally bounded) coherent cohomologies according to Proposition \ref{prop: sheafify has coherent cohomologies}.

We have the following result on the codimension of $Z_m$.

\begin{prop}\label{prop: codim of Zm in sheaf case}
\begin{enumerate}
\item 
If the complex $\underline{F}_X(\mathcal{E})=(\mathfrak{E}^{\bullet},d)$  has cohomologies concentrated in  degrees $\leq n_0$, then we have
			$Z_m=\emptyset$ for any $m\geq n_0$.
\item If the complex $\underline{F}_X(\mathcal{E})=(\mathfrak{E}^{\bullet},d)$  has cohomologies concentrated in degrees $\geq n_0$, then we have
\begin{equation}
  \codim (Z_m)\geq n_0-m,  \text{ for $m\leq n_0-1$.}
\end{equation}
Moreover we have $Z_{m-1}\subseteq Z_m$ for $m\leq n_0-1$.
\end{enumerate}
\end{prop}
\begin{proof}
For any $x\in X$, let $V$ be a neighborhood of $x$ which is sufficiently small.
It is sufficient to prove the proposition for $Z_m\cap V$.

By Proposition \ref{prop: gauge equivalence to flat} on $V$ we have a gauge equivalence
\begin{equation}
J: (E^{\bullet}, A^{E^{\bullet}\prime\prime})|_V\overset{\sim}{\to} (E^{\bullet}|_V, v_0+ \overline{\nabla}^{E^{\bullet}|_V\prime\prime})
\end{equation}
where $(E^{\bullet}|_V, v_0+ \overline{\nabla}^{E^{\bullet}|_V\prime\prime})$ is a bounded cochain complex of holomorphic vector bundles on $V$, whose associated cochain complex of locally free $\mathcal{O}_X$-modules are denoted by $(\overline{\mathfrak{E}}|_V^{\bullet},v_0)$. Notice that $v_0$ is unchanged under the gauge equivalence.

By Proposition \ref{prop: sheafify morphisms} $J$ induces a degreewise isomorphism 
\begin{equation}
\underline{F}_V(J): (\mathfrak{E}|_V^{\bullet},d)\to  \underline{F}_V(E^{\bullet}|_V, v_0+ \overline{\nabla}^{E^{\bullet}|_V\prime\prime}).
\end{equation}
Moreover by the construction of $\underline{F}_V$ as in \eqref{eq: sheafify of cohesive module},  $\underline{F}_V(E^{\bullet}|_V, v_0+ \overline{\nabla}^{E^{\bullet}|_V\prime\prime})$ is quasi-isomorphic to $(\overline{\mathfrak{E}}|_V^{\bullet},v_0)$. We thus obtain a quasi-isomorphism
\begin{equation}
(\mathfrak{E}|_V^{\bullet},d)\overset{\sim}{\to}(\overline{\mathfrak{E}}|_V^{\bullet},v_0).
\end{equation}

For Part 1,
we know $(\overline{\mathfrak{E}}|_V^{\bullet},v_0)$   has cohomologies concentrated in  degrees $\leq n_0$. By definition $Z_m\cap V$ is the subset of $V$ of points such that $v_0: \overline{\mathfrak{E}}|_V^m\to \overline{\mathfrak{E}}|_V^{m+1}$ does not obtain its maximal rank. So it is clear that $Z_m\cap V=\emptyset$ for $m\geq n_0$.

For Part 2, let $n_1$ be the minimal degree such that $\overline{\mathfrak{E}}|_V^{n_1}\neq 0$.  Since $(\overline{\mathfrak{E}}|_V^{\bullet},v_0)$   has cohomologies concentrated in  degrees $\geq n_0$, the sequence of locally free sheaves
\begin{equation}
0\to  \overline{\mathfrak{E}}|_V^{n_1}\overset{v_0}{\longrightarrow}\ldots \overset{v_0}{\longrightarrow} \overline{\mathfrak{E}}|_V^{n_0}
\end{equation}
is exact. The result then follows from the same argument as in the proof of \cite[Theorem 20.9 and Corollary 20.12]{eisenbud1995commutative}. See also \cite[Section 2.7]{larkang2019comparison}.
\end{proof}

\begin{prop}\label{prop: Z contained support F}
For any $l$ we have
 \begin{equation}\label{eq: Z contained support F}
Z_{l-1}\subseteq \supp \mathfrak{H}^{l}(\mathfrak{E}^{\bullet},d),
\end{equation}
where $\mathfrak{H}^{l}(\mathfrak{E}^{\bullet},d)$ is the $l$th cohomology sheaf of the complex $\underline{F}_X(\mathcal{E})=(\mathfrak{E}^{\bullet},d)$.
\end{prop}
\begin{proof}
We first prove the following lemma which is a special case of Proposition \ref{prop: Z contained support F}.
\begin{lemma}\label{lemma: Z contained support F special}
If the complex $\underline{F}_X(\mathcal{E})=(\mathfrak{E}^{\bullet},d)$  has cohomologies concentrated in degrees $\geq n_0$, then we have 
 \begin{equation}\label{eq: Z contained support F special}
Z_{n_0-1}\subseteq \supp \mathfrak{H}^{n_0}(\mathfrak{E}^{\bullet},d).
\end{equation}
\end{lemma}
\begin{proof}[Proof of Lemma \ref{lemma: Z contained support F special}]
Recall $Z_{n_0-1}$ consists  of $x\in X$ such that $v_0^{n_0-1}: E^{n_0-1}\to E^{n_0}$ does not get its maximal rank. On the other hand  we consider $v_0^{n_0}: E^{n_0}\to E^{n_0+1}$, since $\dim \ker v_0^{n_0}$ is a upper semicontinuous function on $X$, it is clear that $\mathfrak{H}^{n_0}(\mathfrak{E}^{\bullet},d)_x\neq 0$ at such $x$. Hence we get the inclusion in \eqref{eq: Z contained support F special}.
\end{proof}

Now we come back to the general case. By the same argument as in the proof of Proposition \ref{prop: codim of Zm in sheaf case}, for any $x\in X$, there exists  a neighborhood $V$ of $x$ which is sufficiently small such that we can consider $(E^{\bullet}|_V, v_0)$ as a bounded cochain complex of holomorphic vector bundles. 

Consider the holomorphic map $v_0^{l-1}: E^{l-1}|_V\to E^{l}|_V$. Then $\ker v_0^{l-1}$ is a coherent sheaf on $V$, hence by Syzygy, it has a bounded locally free resolution if $V$ is sufficiently small, i.e. there exists a bounded complex of holomorphic vector bundles 
\begin{equation}
0\to \tilde{E}^{N}\overset{\tilde{v}_0}{\to}\ldots \overset{\tilde{v}_0}{\to}\tilde{E}^{l-2}
\end{equation}
on $V$ together with a map of $\mathcal{O}_X$-modules $\eta: \tilde{E}^{l-2}\to \ker v_0^{l-1}$ such that the  complex
\begin{equation}
0\to \tilde{E}^{N}\overset{\tilde{v}_0}{\to}\ldots \overset{\tilde{v}_0}{\to}\tilde{E}^{l-2}\overset{\eta}{\to} \ker v_0^{l-1}\to 0
\end{equation}
is acyclic. Now let $i:  \ker v_0^{l-1}\hookrightarrow E^{l-1}|_V$ be the embedding. The bounded complex of holomorphic vector bundles
\begin{equation}
0\to \tilde{E}^{N}\overset{\tilde{v}_0}{\to}\ldots \overset{\tilde{v}_0}{\to}\tilde{E}^{l-2}\overset{i\circ \eta}{\to} E^{l-1}|_V\overset{v_0^{l-1}}{\to} E^{l}|_V \overset{v_0^{l}}{\to}\ldots
\end{equation}
has cohomologies concentrated in degrees $\geq l$, so by Lemma \ref{lemma: Z contained support F special} we have 
\begin{equation}\label{eq: Z contained support F V}
Z_{l-1}\cap V\subseteq \supp \mathfrak{H}^{l}(\mathfrak{E}^{\bullet},d)\cap V.
\end{equation}
Since \eqref{eq: Z contained support F V} holds for any $V$, we get \eqref{eq: Z contained support F}.
\end{proof}

\begin{coro}\label{coro: vanishing of R(U) sheaf case}
For a Hermitian cohesive module $\mathcal{E}= (E^{\bullet}, A^{E^{\bullet}\prime\prime})$ on $X$ and  $\underline{F}_X(\mathcal{E})=(\mathfrak{E}^{\bullet},d)$.
\begin{enumerate}
\item If the complex $\underline{F}_X(\mathcal{E})=(\mathfrak{E}^{\bullet},d)$  has cohomologies concentrated in degrees $\leq n_0$, then for any $q$ and any $k\geq n_0$, we have
\begin{equation}\label{eq: vanishing of R(U) sheaf case 1}
R(U^{\mathcal{E}})_{q \to k}=0.
\end{equation}
\item If the complex $\underline{F}_X(\mathcal{E})=(\mathfrak{E}^{\bullet},d)$  has cohomologies concentrated in degrees $\geq n_0$, then for any $q\leq n_0-1$ and any $k$, we have
\begin{equation}\label{eq: vanishing of R(U) sheaf case 2}
R(U^{\mathcal{E}})_{q \to k}=0.
\end{equation}
\item If there exist integers $m_0\geq 1$ and $n_0$ such that for any $q\leq n_0$, the $q$th cohomology sheaf $\mathfrak{H}^{q}(\mathfrak{E}^{\bullet},d)$ either vanishes or satisfies
\begin{equation}
\codim(\supp \mathfrak{H}^{q}(\mathfrak{E}^{\bullet},d))\geq m_0,
\end{equation} 
then for any $q\leq n_0$ and any $k\geq q-m_0+1$ we have
\begin{equation}\label{eq: vanishing of R(U) sheaf case 3}
R(U^{\mathcal{E}})_{q \to k}=0.
\end{equation}
In particular if $\codim(\supp \mathfrak{H}^{q}(\mathfrak{E}^{\bullet},d))\geq m_0$ for any $q$, then \eqref{eq: vanishing of R(U) sheaf case 3} holds for any $q$ and any $k\geq q-m_0+1$.
\end{enumerate}
\end{coro}
\begin{proof}
Part 1 and 2 are direct consequences of Proposition \ref{prop: vanishing of R(U)} and Proposition \ref{prop: codim of Zm in sheaf case}. Part 3 is also a consequence of Proposition \ref{prop: vanishing of R(U)} and Proposition \ref{prop: codim of Zm in sheaf case}, and Proposition  \ref{prop: Z contained support F}.
\end{proof}

\subsection{Vanishing of $\tilde{R}^{\mathcal{E}}$}
To study the vanishing of $\tilde{R}^{\mathcal{E}}$ we first study $u=\sigma(\id_{E^{\bullet}}+ A^{E^{\bullet}\prime\prime}_{\geq 1}(\sigma))^{-1}$ defined in \eqref{eq: definition of u} in more details. 

We first notice that $u$ is a smooth form on $X \backslash Z$. We define another smooth form $Q$ on $X \backslash Z$ as
\begin{equation}\label{eq: definition of Q}
Q:=\id_{E^{\bullet}}-v_0(\sigma).
\end{equation}

We have the following result on $u^{\mathcal{E}}$.

\begin{lemma}\label{lemma: Au}
On  $X \backslash Z$ we have
\begin{equation}\label{eq: Au}
A^{E^{\bullet}\prime\prime}(u^{\mathcal{E}})= \id_{E^{\bullet}}-Q(\id_{E^{\bullet}}+ A^{E^{\bullet}\prime\prime}_{\geq 1}(\sigma))^{-1}+u^{\mathcal{E}} A^{E^{\bullet}\prime\prime}_{\geq 1}(Q)(\id_{E^{\bullet}}+ A^{E^{\bullet}\prime\prime}_{\geq 1}(\sigma))^{-1}.
\end{equation}
\end{lemma}
\begin{proof}
By definition we know
\begin{equation}\label{eq: first expansion of Au}
\begin{split}
&A^{E^{\bullet}\prime\prime}(u)= A^{E^{\bullet}\prime\prime}(\sigma)(\id_{E^{\bullet}}+ A^{E^{\bullet}\prime\prime}_{\geq 1}(\sigma))^{-1}-\sigma A^{E^{\bullet}\prime\prime}((\id_{E^{\bullet}}+ A^{E^{\bullet}\prime\prime}_{\geq 1}(\sigma))^{-1})\\
=& A^{E^{\bullet}\prime\prime}(\sigma)(\id_{E^{\bullet}}+ A^{E^{\bullet}\prime\prime}_{\geq 1}(\sigma))^{-1}\\
&+\sigma  (\id_{E^{\bullet}}+ A^{E^{\bullet}\prime\prime}_{\geq 1}(\sigma))^{-1} A^{E^{\bullet}\prime\prime}((\id_{E^{\bullet}}+ A^{E^{\bullet}\prime\prime}_{\geq 1}(\sigma)))(\id_{E^{\bullet}}+ A^{E^{\bullet}\prime\prime}_{\geq 1}(\sigma))^{-1}\\
=& A^{E^{\bullet}\prime\prime}(\sigma)(\id_{E^{\bullet}}+ A^{E^{\bullet}\prime\prime}_{\geq 1}(\sigma))^{-1}+ u^{\mathcal{E}} A^{E^{\bullet}\prime\prime}(\id_{E^{\bullet}}+ A^{E^{\bullet}\prime\prime}_{\geq 1}(\sigma))(\id_{E^{\bullet}}+ A^{E^{\bullet}\prime\prime}_{\geq 1}(\sigma))^{-1}
\end{split}
\end{equation}
We know
\begin{equation}
A^{E^{\bullet}\prime\prime}(\sigma)=v_0(\sigma)+A^{E^{\bullet}\prime\prime}_{\geq 1}(\sigma)
=\id_{E^{\bullet}}-Q(\id_{E^{\bullet}}+A^{E^{\bullet}\prime\prime}_{\geq 1}(\sigma)
\end{equation}
hence the first term on the right hand side of \eqref{eq: first expansion of Au} becomes
\begin{equation}
A^{E^{\bullet}\prime\prime}(\sigma)(\id_{E^{\bullet}}+ A^{E^{\bullet}\prime\prime}_{\geq 1}(\sigma))^{-1}=\id_{E^{\bullet}}-Q(\id_{E^{\bullet}}+ A^{E^{\bullet}\prime\prime}_{\geq 1}(\sigma))^{-1}.
\end{equation}

Moreover $A^{E^{\bullet}\prime\prime}(\id_{E^{\bullet}})=0$ and $A^{E^{\bullet}\prime\prime}A^{E^{\bullet}\prime\prime}=0$. Hence
\begin{equation}
\begin{split}
A^{E^{\bullet}\prime\prime}(\id_{E^{\bullet}}&+ A^{E^{\bullet}\prime\prime}_{\geq 1}(\sigma))=A^{E^{\bullet}\prime\prime}( A^{E^{\bullet}\prime\prime}(\sigma)-v_0(\sigma))\\
&=A^{E^{\bullet}\prime\prime}(-v_0(\sigma))\\
&=A^{E^{\bullet}\prime\prime}(Q-\id_{E^{\bullet}})\\
&=A^{E^{\bullet}\prime\prime}(Q).
\end{split}
\end{equation}
Moreover since 
\begin{equation}
v_0(Q)=v_0(\id_{E^{\bullet}}-v_0(\sigma))=v_0(\id_{E^{\bullet}})-v_0(v_0(\sigma))=0
\end{equation}
we get $A^{E^{\bullet}\prime\prime}(Q)=A^{E^{\bullet}\prime\prime}_{\geq 1}(Q)$.
So the second term on the right hand side of \eqref{eq: first expansion of Au} becomes
\begin{equation}
u A^{E^{\bullet}\prime\prime}(\id_{E^{\bullet}}+ A^{E^{\bullet}\prime\prime}_{\geq 1}(\sigma))(\id_{E^{\bullet}}+ A^{E^{\bullet}\prime\prime}_{\geq 1}(\sigma))^{-1}=u^{\mathcal{E}} A^{E^{\bullet}\prime\prime}_{\geq 1}(Q)(\id_{E^{\bullet}}+ A^{E^{\bullet}\prime\prime}_{\geq 1}(\sigma))^{-1}.
\end{equation}
We thus obtain \eqref{eq: Au}.
\end{proof}

\begin{prop}\label{prop: Au=id}
For a fixed integer $n_0$,
If the cohomology sheaf $\mathfrak{H}^{\bullet}(\mathfrak{E}^{\bullet},d)$ of the complex $\underline{F}_X(\mathcal{E})=(\mathfrak{E}^{\bullet},d)$ is such that  for all $q \leq n_0$, either $\mathfrak{H}^q(\mathfrak{E}^{\bullet},d)=0$  or 
\begin{equation}
\codim ( \supp \mathfrak{H}^q(\mathfrak{E}^{\bullet},d))\geq 1,
\end{equation}
 then for any $q\leq n_0$ and any $k$, we have
\begin{equation}
\big(\id_{E^{\bullet}}-A^{E^{\bullet}\prime\prime}(u)\big)_{q \to k}=0
\end{equation}
on $X \backslash Z$.  
\end{prop}
\begin{proof}
By \eqref{eq: Au} an the identity
$$
(\id_{E^{\bullet}}+ A^{E^{\bullet}\prime\prime}_{\geq 1}(\sigma))^{-1}=\id_{E^{\bullet}}-A^{E^{\bullet}\prime\prime}_{\geq 1}(\sigma)+(A^{E^{\bullet}\prime\prime}_{\geq 1}(\sigma))^2-\ldots,
$$ 
it is sufficient to prove that on $X \backslash Z$ we have
\begin{equation}\label{eq: Q Asigma=0}
\big(Q( A^{E^{\bullet}\prime\prime}_{\geq 1}(\sigma))^j\big)_{q \to k}=0
\end{equation}
and 
\begin{equation}\label{eq: AQ Asigma=0}
\big(u^{\mathcal{E}} A^{E^{\bullet}\prime\prime}_{\geq 1}(Q)( A^{E^{\bullet}\prime\prime}_{\geq 1}(\sigma))^j\big)_{q \to k}=0
\end{equation}
for any $j\geq 0$ and any $q\leq n_0$. 

Now since $(\mathfrak{E}^{\bullet},d)$  has cohomologies concentrated in degrees $\geq n_0+1$ or 
$$
\codim ( \supp \mathfrak{H}^l(\mathfrak{E}^{\bullet},d))\geq 1
$$
 for $l \leq n_0$, via the same argument as in the proof of Proposition \ref{prop: codim of Zm in sheaf case} we know the complex
$$
(E^{\bullet},v_0)|_{X \backslash Z}
$$
is exact at degree $\leq n_0-1$. Then it is easy to see that on $X \backslash Z$ we have
\begin{equation}
Q_{(r,s)\to (r,s)}=0 \text{, and }  A^{E^{\bullet}\prime\prime}_{\geq 1}(Q)_{(r,s)\to (r,s)}=0
\end{equation}
for any $s<n_0$. Since $A^{E^{\bullet}\prime\prime}_{\geq 1}(\sigma))^j$ does not increase the degree on $E^{\bullet}$, we get \eqref{eq: Q Asigma=0} and \eqref{eq: AQ Asigma=0}.
\end{proof}

\begin{coro}\label{coro: tilde R vanish}
 If there exists an integer  $n_0$ such that for any $q\leq n_0$, the $q$th cohomology sheaf $\mathfrak{H}^{q}(\mathfrak{E}^{\bullet},d)$ either vanishes or satisfies
\begin{equation}
\codim ( \supp \mathfrak{H}^{q}(\mathfrak{E}^{\bullet},d))\geq 1,
\end{equation}
then for any $q\leq n_0$ and any $k$, we have
\begin{equation}
\tilde{R}^{\mathcal{E}}_{q \to k}=0.
\end{equation}
In particular if $\codim ( \supp \mathfrak{H}^{l}(\mathfrak{E}^{\bullet},d))\geq 1$ for any $l$, then $
\tilde{R}^{\mathcal{E}}=0$.

\end{coro}
\begin{proof}
Recall \eqref{eq: tilde R and R further} gives us $\tilde{R}^{\mathcal{E}}=\id_{E^{\bullet}}-\mathbf{1}_{X \backslash Z}A^{E^{\bullet}\prime\prime}(U^{\mathcal{E}})$.
By \eqref{eq: 1X/Z|X/Z} we know that
\begin{equation}
\mathbf{1}_{X \backslash Z}A^{E^{\bullet}\prime\prime}(U^{\mathcal{E}})|_{X \backslash Z}=A^{E^{\bullet}\prime\prime}(U^{\mathcal{E}})|_{X \backslash Z}=A^{E^{\bullet}\prime\prime}(U^{\mathcal{E}}|_{X \backslash Z})=A^{E^{\bullet}\prime\prime}(u^{\mathcal{E}}).
\end{equation}
By Lemma \ref{lemma: tilde R is LASM}, $\tilde{R}^{\mathcal{E}}$ is the unique locally almost semimeromorphic (LASM) extension of $\id_{E^{\bullet}}-A^{E^{\bullet}\prime\prime}(u^{\mathcal{E}})$ to $X$. Now the claims follow from Proposition \ref{prop: Au=id}. 
\end{proof}

The following corollary, which is one of the main result in this section, gives the vanishing result for the residue current $R^{\mathcal{E}}$.

\begin{coro}\label{coro: R vanish}
\begin{enumerate}
\item If the complex $\underline{F}_X(\mathcal{E})=(\mathfrak{E}^{\bullet},d)$  has cohomologies concentrated in degrees $\geq n_0$, then for any $q\leq n_0-1$ and any $k$, we have
\begin{equation}
R^{\mathcal{E}}_{q \to k}=0.
\end{equation}
\item If there exist integers $m_0\geq 1$ and $n_0$ such that for any $q\leq n_0$, the $q$th cohomology sheaf $\mathfrak{H}^{q}(\mathfrak{E}^{\bullet},d)$ either vanishes or satisfies
\begin{equation}
\codim(\supp \mathfrak{H}^{q}(\mathfrak{E}^{\bullet},d))\geq m_0,
\end{equation} 
then for any $q\leq n_0$ and any $k\geq q-m_0+1$ we have
\begin{equation}\label{eq: R vanishes}
R^{\mathcal{E}}_{q \to k}=0.
\end{equation}
In particular if $\codim(\supp \mathfrak{H}^{q}(\mathfrak{E}^{\bullet},d))\geq m_0$ for any $q$, then \eqref{eq: R vanishes} holds for any $q\leq n_0$ and any $k\geq q-m_0+1$.
\end{enumerate}
\end{coro}
\begin{proof}
They are direct consequences of \eqref{eq: tilde R and R}, Corollary \ref{coro: vanishing of R(U) sheaf case}, and Corollary \ref{coro: tilde R vanish}.
\end{proof}

We then have the following precise form of the duality principle.
\begin{thm}\label{thm: duality principle sheaf case}
Let $\mathcal{E}= (E^{\bullet}, A^{E^{\bullet}\prime\prime})$  be a Hermitian cohesive module on a complex manifold $X$. If the complex $\underline{F}_X(\mathcal{E})=(\mathfrak{E}^{\bullet},d)$  has cohomologies concentrated in degrees $\geq n_0$, Let 
$$
s\in  \Gamma(X, \mathfrak{E}^k)=\bigoplus_{p+q=k}\Gamma(X, \Omega^{0,p}\hat{\otimes} E^q)
$$
be such that $A^{E^{\bullet}\prime\prime}(s)=0$. 
\begin{enumerate}
\item If $k\leq n_0-1$, then we must have  $R^{\mathcal{E}}(s)=0$ and  a 
$$
t\in \Gamma(X, \mathfrak{E}^{k-1})=\bigoplus_{p+q=k-1}\Gamma(X, \Omega^{0,p}\hat{\otimes} E^q)
$$
such that $A^{E^{\bullet}\prime\prime}(t)=s$.
\item If $k=n_0$, then  there  exists a 
$$
t\in  \Gamma(X, \mathfrak{E}^{k-1})=\bigoplus_{p+q=k-1}\Gamma(X, \Omega^{0,p}\hat{\otimes} E^q)
$$
such that $A^{E^{\bullet}\prime\prime}(t)=s$ if and only if $R^{\mathcal{E}}(s)=0$.
\end{enumerate}
\end{thm}
\begin{proof}
Both statements are direct consequences of Theorem \ref{thm: duality principle} and Corollary \ref{coro: R vanish} Part 1.
\end{proof}

\begin{rmk}
 See \cite[Theorem 1.1]{Johansson2023residue} for a similar result in the framework of twisting cochains.
\end{rmk}

\begin{rmk}
In Theorem \ref{thm: duality principle sheaf case}, even if we make the stronger assumption that $(\mathfrak{E}^{\bullet},d)$  has cohomology concentrated in degree $=n_0$, for $k\geq n_0+1$, there may still exist 
$$
s\in \Gamma(X, \mathfrak{E}^k)=\bigoplus_{p+q=k}\Gamma(X, \Omega^{0,p}\hat{\otimes} E^q)
$$
 and
$$
t\in \Gamma(X, \mathfrak{E}^{k-1})=\bigoplus_{p+q=k-1}\Gamma(X, \Omega^{0,p}\hat{\otimes} E^q)
$$ 
such that  $s=A^{E^{\bullet}\prime\prime}(t)$ but $R^{\mathcal{E}}(s)\neq 0$. 

For example let $E^{\bullet}=\underline{\mathbb{C}}$ be the trivial line bundle concentrated in degree $0$. Let $\nabla^{E^{\bullet}\prime\prime}=\dbar$ and all $v_i$'s be $0$. Then $u^{\mathcal{E}}\equiv 0$ hence $U^{\mathcal{E}}\equiv 0$ and \eqref{eq: RE} give $R^{\mathcal{E}}=\id_{\underline{\mathbb{C}}}$.

Now consider a non-holomorphic $C^{\infty}$-function $t$ on $X$. We have
$$
A^{E^{\bullet}\prime\prime}(t)=\dbar(t)\neq 0.
$$
Let $s=\dbar(t)\in C^{\infty}(X, \overline{T^{*}X})\subset \oplus_{p+q=1}\Gamma(X, \Omega^{0,p}\hat{\otimes} E^q)$. We have $s=A^{E^{\bullet}\prime\prime}(t)$ but $R^{\mathcal{E}}(s)=s\neq 0$.
\end{rmk}

\section{A comparison formula for residue currents of Hermitian cohesive modules}\label{section: comparison formula}
\subsection{A comparison formula}
In this section we generalize the results in \cite{larkang2019comparison}. Let $\mathcal{E}= (E^{\bullet}, A^{E^{\bullet}\prime\prime})$ and $\mathcal{F}= (F^{\bullet}, A^{F^{\bullet}\prime\prime})$ be two Hermitian cohesive modules on $X$ and $\phi: \mathcal{E}\to \mathcal{F}$ be a closed degree $0$ morphism. 

Let $U^{\mathcal{E}}$, $R^{\mathcal{E}}$, $R(U^{\mathcal{E}})$, $\tilde{R}^{\mathcal{E}}$, and $U^{\mathcal{F}}$, $R^{\mathcal{F}}$, $R(U^{\mathcal{F}})$, $\tilde{R}^{\mathcal{F}}$ be currents defined in Section \ref{subsection: residue current of a cohesive module definition} associated with $\mathcal{E}$ and $\mathcal{F}$ respectively. Since both $U^{\mathcal{E}}$ and $U^{\mathcal{F}}$ are locally almost semimeromorphic (LASM), and $\phi$ is smooth, by Proposition \ref{prop: almost semimeromorphic currents form an algebra} we can define the product current $U^{\mathcal{F}}\phi U^{\mathcal{E}}$, whose differential  is
$$
D^{\mathcal{E},\mathcal{F}}(U^{\mathcal{F}}\phi U^{\mathcal{E}}):=A^{F^{\bullet}\prime\prime}U^{\mathcal{F}}\phi U^{\mathcal{E}}- U^{\mathcal{F}}\phi U^{\mathcal{E}}A^{E^{\bullet}\prime\prime}
$$
Let $Z\subset X$ be the  of points at which $E^i\to E^{i+1}$ or $F^j\to F^{j+1}$ does not obtain its maximal rank for some $i$ or $j$. Then $Z$ is  still an analytic subvariety of $X$ with codimension $\geq 1$. As before we define the residue of $U^{\mathcal{F}}\phi U^{\mathcal{E}}$ as
\begin{equation}\label{eq: R(UF phi UE)}
R(U^{\mathcal{F}}\phi U^{\mathcal{E}}):=D^{\mathcal{E},\mathcal{F}}(U^{\mathcal{F}}\phi U^{\mathcal{E}})-\mathbf{1}_{X \backslash Z} D^{\mathcal{E},\mathcal{F}}(U^{\mathcal{F}}\phi U^{\mathcal{E}})
\end{equation}

We  define the current $\tilde{M}^{\phi}$  as
\begin{equation}\label{eq: tilde M phi}
\tilde{M}^{\phi}:=\tilde{R}^{\mathcal{F}}\phi U^{\mathcal{E}}-U^{\mathcal{F}}\phi \tilde{R}^{\mathcal{E}}.
\end{equation}
It is clear that  $\tilde{M}^{\phi}$ is a LASM current with total degree $-1$.  We then define the  pseudomeromorphic (PM) current  $M^{\phi}$ as
\begin{equation}\label{eq: M phi}
M^{\phi}:=\tilde{M}^{\phi}+R(U^{\mathcal{F}}\phi U^{\mathcal{E}})
\end{equation}

\begin{thm}\label{thm: comparison formula}
Let $\mathcal{E}= (E^{\bullet}, A^{E^{\bullet}\prime\prime})$ and $\mathcal{F}= (F^{\bullet}, A^{F^{\bullet}\prime\prime})$ be two Hermitian cohesive modules on $X$ and $\phi: \mathcal{E}\to \mathcal{F}$ be a closed degree $0$ morphism. 
The residue currents  $R^{\mathcal{E}}$ and  $R^{\mathcal{F}}$ are related via the morphism $\phi$ in the sense that
\begin{equation}\label{eq: residue current compatible}
R^{\mathcal{F}}\phi-\phi R^{\mathcal{E}}=D^{\mathcal{E},\mathcal{F}}(M^{\phi}).
\end{equation}
\end{thm}
\begin{proof}
Since $D^{\mathcal{E},\mathcal{F}}\phi=0$, we have
\begin{equation}\label{eq: D of U phi U}
\begin{split}
D^{\mathcal{E},\mathcal{F}}(U^{\mathcal{F}}\phi U^{\mathcal{E}})&=A^{F^{\bullet}\prime\prime}(U^{\mathcal{F}})\phi U^{\mathcal{E}}-U^{\mathcal{F}}\phi A^{E^{\bullet}\prime\prime}(U^{\mathcal{E}})\\
&=(\id_{F^{\bullet}}-R^{\mathcal{F}})\phi U^{\mathcal{E}}-U^{\mathcal{F}}\phi  (\id_{E^{\bullet}}-R^{\mathcal{E}})\\
&=\phi U^{\mathcal{E}}-U^{\mathcal{F}}\phi -R^{\mathcal{F}}\phi U^{\mathcal{E}}+U^{\mathcal{F}}\phi R^{\mathcal{E}}.
\end{split}
\end{equation}
Recall that by Proposition \ref{prop: pseudomeromorphic currents closed under multiplication}, the right hand side of \eqref{eq: D of U phi U} is a well-defined PM current.
By \eqref{eq: R and R tilde equal} we further have
\begin{equation}
D^{\mathcal{E},\mathcal{F}}(U^{\mathcal{F}}\phi U^{\mathcal{E}})|_{X \backslash Z}=(\phi U^{\mathcal{E}}-U^{\mathcal{F}}\phi -\tilde{R}^{\mathcal{F}}\phi U^{\mathcal{E}}+U^{\mathcal{F}}\phi \tilde{R}^{\mathcal{E}})|_{X \backslash Z}.
\end{equation}

Notice that $\phi U^{\mathcal{E}}-U^{\mathcal{F}}\phi -\tilde{R}^{\mathcal{F}}\phi U^{\mathcal{E}}+U^{\mathcal{F}}\phi \tilde{R}^{\mathcal{E}}$ is a LASM current on $X$, hence it is a  LASM extension of $D^{\mathcal{E},\mathcal{F}}(U^{\mathcal{F}}\phi U^{\mathcal{E}})|_{X \backslash Z}$. On the other hand $\mathbf{1}_{X \backslash Z} D^{\mathcal{E},\mathcal{F}}(U^{\mathcal{F}}\phi U^{\mathcal{E}})$ is also a LASM extension of $D^{\mathcal{E},\mathcal{F}}(U^{\mathcal{F}}\phi U^{\mathcal{E}})|_{X \backslash Z}$. By the uniqueness of LASM extension as in Lemma \ref{lemma: uniqueness of LASM extension}, we must have
\begin{equation}
\mathbf{1}_{X \backslash Z} D^{\mathcal{E},\mathcal{F}}(U^{\mathcal{F}}\phi U^{\mathcal{E}})=\phi U^{\mathcal{E}}-U^{\mathcal{F}}\phi -\tilde{R}^{\mathcal{F}}\phi U^{\mathcal{E}}+U^{\mathcal{F}}\phi \tilde{R}^{\mathcal{E}}
\end{equation}
hence
\begin{equation}\label{eq: R(U phi U)}
R(U^{\mathcal{F}}\phi U^{\mathcal{E}})=\phi U^{\mathcal{E}}-U^{\mathcal{F}}\phi -\tilde{R}^{\mathcal{F}}\phi U^{\mathcal{E}}+U^{\mathcal{F}}\phi \tilde{R}^{\mathcal{E}}-D^{\mathcal{E},\mathcal{F}}(U^{\mathcal{F}}\phi U^{\mathcal{E}}).
\end{equation}
\eqref{eq: tilde M phi}, \eqref{eq: M phi}, and \eqref{eq: R(U phi U)} give
\begin{equation}
M^{\phi}=\phi U^{\mathcal{E}}-U^{\mathcal{F}}\phi -D^{\mathcal{E},\mathcal{F}}(U^{\mathcal{F}}\phi U^{\mathcal{E}}).
\end{equation}
Therefore we get
\begin{equation}
\begin{split}
D^{\mathcal{E},\mathcal{F}}(M^{\phi})&=D^{\mathcal{E},\mathcal{F}}(\phi U^{\mathcal{E}}-U^{\mathcal{F}}\phi)\\
&=\phi A^{E^{\bullet}\prime\prime}(U^{\mathcal{E}})-A^{F^{\bullet}\prime\prime}U^{\mathcal{F}}\phi\\
&=\phi (\id_{E^{\bullet}}-R^{\mathcal{E}})-(\id_{F^{\bullet}}-R^{\mathcal{F}})\phi\\
&=R^{\mathcal{F}}\phi-\phi R^{\mathcal{E}}
\end{split}
\end{equation}
as expected.
\end{proof}

\begin{coro}\label{coro: homotopic comparison formula}
Let $\mathcal{E}= (E^{\bullet}, A^{E^{\bullet}\prime\prime})$ and $\mathcal{F}= (F^{\bullet}, A^{F^{\bullet}\prime\prime})$ be two Hermitian cohesive modules on $X$. Let $\phi: \mathcal{E}\to \mathcal{F}$ and $\psi: \mathcal{F}\to \mathcal{E}$ be two closed degree $0$ morphisms which are homotopic inverse to each other, i.e. there exists  degree $-1$ morphisms $\tau: \mathcal{E}\to \mathcal{F}$ and $\gamma: \mathcal{F}\to \mathcal{E}$ such that 
\begin{equation}\label{eq: homotopic inverse}
\psi\phi-\id_{E^{\bullet}}=A^{E^{\bullet}\prime\prime}(\tau) \text{, and }\phi\psi-\id_{F^{\bullet}}=A^{F^{\bullet}\prime\prime}(\gamma).
\end{equation}
Then $R^{\mathcal{E}}$ is homotopic to $\psi R^{\mathcal{F}}\phi$ and $R^{\mathcal{F}}$ is homotopic to $\phi R^{\mathcal{E}}\psi$. More precisely, let $M^{\phi}$ and $M^{\psi}$ be currents associated with $\phi$ and $\psi$ as in \eqref{eq: M phi}. Then we have
\begin{equation}\label{eq: homotopic comparison formula}
\begin{split}
&R^{\mathcal{E}}-\psi R^{\mathcal{F}}\phi=A^{E^{\bullet}\prime\prime}(M^{\psi}\phi-R^{\mathcal{E}}\tau),\\
& R^{\mathcal{F}}-\phi R^{\mathcal{E}}\psi=A^{F^{\bullet}\prime\prime}(M^{\phi}\psi-R^{\mathcal{F}}\gamma).
\end{split}
\end{equation}
\end{coro}
\begin{proof}
It is a direct consequence of Theorem \ref{thm: comparison formula} and \eqref{eq: homotopic inverse}.
\end{proof}

\begin{rmk}
See \cite[Theorem 1.3]{Johansson2023residue} for a similar result in the framework of twisting cochains.
\end{rmk}

\begin{rmk}
Theorem \ref{thm: equiv of cats} implies that if $\mathcal{E}$ and $\mathcal{F}$ are two Hermitian cohesive modules on $X$ which are cohesive resolutions of the same object in $D^{\gb}_{\coh}(X)$, then the morphisms $\phi$, $\psi$, $\tau$, and $\gamma$ in \eqref{eq: homotopic inverse} exists. Corollary \ref{coro: homotopic comparison formula} tells us that in this case the residue currents $R^{\mathcal{E}}$ and  $R^{\mathcal{F}}$ are essentially the same.
\end{rmk}

\subsection{Vanishing of $M^{\phi}$}
We have the following results on the vanishing of $R(U^{\mathcal{F}}\phi U^{\mathcal{E}})$, $\tilde{M}^{\phi}$, and $M^{\phi}$.

\begin{prop}\label{prop: vanishing of R(UphiU)}
Let $\phi: \mathcal{E}\to \mathcal{F}$ be a closed degree $0$ morphism between Hermitian cohesive modules on $X$ and $R(U^{\mathcal{F}}\phi U^{\mathcal{E}})$ be as in \eqref{eq: R(UF phi UE)}. For any $k\geq q-1$ we have
\begin{equation}\label{eq: vanishing of R(UphiU) trivial}
R(U^{\mathcal{F}}\phi U^{\mathcal{E}})_{q \to k}=0.
\end{equation}

Moreover if there exists a pair of integers $l$, $q$ such that $l\leq q-2$ and the subvarieties $Z^{\mathcal{E}}_i$'s and $Z^{\mathcal{F}}_i$'s satisfy
\begin{equation}
\begin{split}
&\codim Z^{\mathcal{E}}_m\geq q-m+1, \text{ for } l+1\leq m\leq q-1, \text{ and}\\
&\codim Z^{\mathcal{F}}_m\geq q-m, \text{ for } l\leq m \leq  q-2.
\end{split}
\end{equation}
Then  for any $p\geq 0$ and $k\geq l$ we have
\begin{equation}\label{eq: vanishing of R(UphiU)}
R(U^{\mathcal{F}}\phi U^{\mathcal{E}})_{q \to k}=0.
\end{equation}
\end{prop}
\begin{proof}
Since $\phi$ is of degree $0$, it does not increase the degree on $E^{\bullet}$. Hence the proof is similar to that of Proposition \ref{prop: vanishing of R(U)} and is left to the readers.
\end{proof}

\begin{rmk}
Proposition \ref{prop: vanishing of R(UphiU)} is a generalization of \cite[Proposition 3.6]{larkang2019comparison}. See \cite[Proposition 5.2]{Johansson2023residue}  for a similar result in the framework of twisting cochains.
\end{rmk}

\begin{coro}\label{coro: vanishing of R(UphiU) sheaf case}
Let $\phi: \mathcal{E}\to \mathcal{F}$ be a closed degree $0$ morphism between Hermitian cohesive modules on $X$ and $R(U^{\mathcal{F}}\phi U^{\mathcal{E}})$ be as in \eqref{eq: R(UF phi UE)}. We consider complexes of sheaves $\underline{F}_X(\mathcal{E})=(\mathfrak{E}^{\bullet},d)$ and $\underline{F}_X(\mathcal{F})=(\mathfrak{F}^{\bullet},d)$.
\begin{enumerate}
\item If $(\mathfrak{E}^{\bullet},d)$ has cohomologies concentrated in degrees $\leq n_0+1$ and  $(\mathfrak{F}^{\bullet},d)$ has cohomologies concentrated in degrees $\leq n_0$, then for any $q$ and any $k\geq n_0$, we have
\begin{equation}\label{eq: vanishing of R(UphiU) sheaf case 1}
R(U^{\mathcal{F}}\phi U^{\mathcal{E}})_{q \to k}=0.
\end{equation}
\item If $(\mathfrak{E}^{\bullet},d)$ has cohomologies concentrated in degrees $\geq n_0$ and  $(\mathfrak{F}^{\bullet},d)$ has cohomologies concentrated in degrees $\geq n_0-1$, then for any $q\leq n_0-1$ and any $k$, we have
\begin{equation}\label{eq: vanishing of R(UphiU) sheaf case 2}
R(U^{\mathcal{F}}\phi U^{\mathcal{E}})_{q \to k}=0.
\end{equation}
\item If there exist integers $m_0\geq 1$ and $n_0$ such that for any $q\leq n_0$ we have $\mathfrak{H}^{q}(\mathfrak{E}^{\bullet},d)$ either vanishes or satisfies
\begin{equation}\label{eq: codim of cohomology of suppE geq m0}
\codim(\supp \mathfrak{H}^{q}(\mathfrak{E}^{\bullet},d))\geq m_0,
\end{equation} 
and $ \mathfrak{H}^{q-1}(\mathfrak{F}^{\bullet},d)$ either vanishes or satisfies 
\begin{equation}\label{eq: codim of cohomology of suppF geq m0}
\codim(\supp \mathfrak{H}^{q-1}(\mathfrak{F}^{\bullet},d))\geq m_0,
\end{equation}
then for any $q\leq n_0$ and any $k\geq q-m_0$ we have
\begin{equation}\label{eq: vanishing of R(UphiU) sheaf case 3}
R(U^{\mathcal{F}}\phi U^{\mathcal{E}})_{q \to k}=0.
\end{equation}
In particular if \eqref{eq: codim of cohomology of suppE geq m0} and \eqref{eq: codim of cohomology of suppF geq m0} hold for any $q$, then \eqref{eq: vanishing of R(UphiU) sheaf case 3} holds for any $q$ and any $k\geq q-m_0$.
\end{enumerate}
\end{coro}
\begin{proof}
They are direct consequences of Proposition \ref{prop: vanishing of R(UphiU)}, Proposition \ref{prop: codim of Zm in sheaf case}, and Proposition \ref{prop: Z contained support F}.
\end{proof}

\begin{prop}\label{prop: vanishing of tildeM}
Let $\phi: \mathcal{E}\to \mathcal{F}$ be a closed degree $0$ morphism between Hermitian cohesive modules on $X$ and $R(U^{\mathcal{F}}\phi U^{\mathcal{E}})$ be as in \eqref{eq: R(UF phi UE)}. 
 If there exists an integer $n_0$ such that for any $q\leq n_0$ we have $\mathfrak{H}^{q}(\mathfrak{E}^{\bullet},d)$ either vanishes or satisfies
\begin{equation}\label{eq: codim of cohomology of suppE geq 1}
\codim(\supp \mathfrak{H}^{q}(\mathfrak{E}^{\bullet},d))\geq 1,
\end{equation} 
and $ \mathfrak{H}^{q-1}(\mathfrak{F}^{\bullet},d)$ either vanishes or satisfies 
\begin{equation}\label{eq: codim of cohomology of suppF geq 1}
\codim(\supp \mathfrak{H}^{q-1}(\mathfrak{F}^{\bullet},d))\geq 1,
\end{equation}
then for any $q\leq n_0$ and any $k$ we have
\begin{equation}
\tilde{M}^{\phi}_{q \to k}=0.
\end{equation}
In particular if we have
\begin{equation}
\begin{split}
&\codim ( \supp \mathfrak{H}^{q}(\mathfrak{E}^{\bullet},d))\geq 1\text{, and }\\
&\codim ( \supp \mathfrak{H}^{q}(\mathfrak{F}^{\bullet},d))\geq 1,
\end{split}
\end{equation}
for any $q$, then $\tilde{M}^{\phi}=0$.
\end{prop}
\begin{proof}
By \eqref{eq: tilde M phi}, $
\tilde{M}^{\phi}=\tilde{R}^{\mathcal{F}}\phi U^{\mathcal{E}}-U^{\mathcal{F}}\phi \tilde{R}^{\mathcal{E}}$.
Notice that $U^{\mathcal{E}}$ lowers the $E^{\bullet}$ degree and $\phi$ does not increase the $E^{\bullet}$ degree.
Now the claims are consequences of 
Corollary \ref{coro: tilde R vanish}.
\end{proof}

The following corollary, which is the main result in this subsection, gives the vanishing result for $M^{\phi}$.

\begin{coro}\label{coro: vanishing of M}
Let $\phi: \mathcal{E}\to \mathcal{F}$ be a closed degree $0$ morphism between Hermitian cohesive modules on $X$ and $R(U^{\mathcal{F}}\phi U^{\mathcal{E}})$ be as in \eqref{eq: R(UF phi UE)}.

\begin{enumerate} 
\item If $(\mathfrak{E}^{\bullet},d)$ has cohomologies concentrated in degrees $\geq n_0$ and  $(\mathfrak{F}^{\bullet},d)$ has cohomologies concentrated in degrees $\geq n_0-1$,  then for any $q\leq n_0-1$ and any $k$, we have
\begin{equation}\label{eq: vanishing of R(UphiU) sheaf case 1}
M^{\phi}_{q \to k}=0.
\end{equation}
\item  If there exist integers $m_0\geq 1$ and $n_0$ such that for any $q\leq n_0$ we have $\mathfrak{H}^{q}(\mathfrak{E}^{\bullet},d)$ either vanishes or satisfies
\begin{equation}\label{eq: codim of cohomology of suppE geq m0 again}
\codim(\supp \mathfrak{H}^{q}(\mathfrak{E}^{\bullet},d))\geq m_0,
\end{equation} 
and $ \mathfrak{H}^{q-1}(\mathfrak{F}^{\bullet},d)$ either vanishes or satisfies 
\begin{equation}\label{eq: codim of cohomology of suppF geq m0 again}
\codim(\supp \mathfrak{H}^{q-1}(\mathfrak{F}^{\bullet},d))\geq m_0,
\end{equation}
then for any $q\leq n_0$ and any $k\geq q-m_0$ we have
\begin{equation}\label{eq: vanishing of R(UphiU) sheaf case 2}
M^{\phi}_{q \to k}=0.
\end{equation}
In particular if \eqref{eq: codim of cohomology of suppE geq m0 again} and \eqref{eq: codim of cohomology of suppF geq m0 again} hold for any $q$, then \eqref{eq: vanishing of R(UphiU) sheaf case 2} holds for any $q$ and any $k\geq q-m_0$.
\end{enumerate}
\end{coro}
\begin{proof}
It is a direct consequence of Corollary \ref{coro: vanishing of R(UphiU) sheaf case} and Proposition \ref{prop: vanishing of tildeM}.
\end{proof}

\section{A generalized Poincar\'{e}-Lelong formula}\label{section: Poincare-Lelong}
\subsection{Some definitions and notations}
\subsubsection{Cycles}
For a coherent sheaf $\mathfrak{F}$ on $X$, the \emph{cycle} of $\mathfrak{F}$ is defined to be the current
\begin{equation}\label{eq: cycle of a coherent sheaf}
[\mathfrak{F}]:=\sum_i m_i[Z_i],
\end{equation}
where the $Z_i$'s are the irreducible components of $\supp \mathfrak{F}$, and $m_i$ is the \emph{multiplicity} of $Z_i$ in $\mathfrak{F}$. See \cite[\href{https://stacks.math.columbia.edu/tag/02QV}{Tag 02QV}]{stacks-project} for details. 

We say that a coherent sheaf $\mathfrak{F}$ has \emph{pure codimension} $p$ if $\supp\mathfrak{F}$ has pure codimension $p$, i.e. every irreducible component of $\supp\mathfrak{F}$ has the same codimension $p$. If $\mathfrak{F}$ is not pure, let $[\mathfrak{F}]_p$ denote the sum of codimension $p$ components of $[\mathfrak{F}]$.

Now let $(\mathfrak{F}^{\bullet},d)$ be a bounded complex of $\mathcal{O}_X$-modules with coherent cohomologies. We define the \emph{cycle} of $(\mathfrak{F}^{\bullet},d)$ to be the current
\begin{equation}\label{eq: cycle of a cochain complex}
[(\mathfrak{F}^{\bullet},d)]:=\sum_l(-1)^l[\mathfrak{H}^l(\mathfrak{F}^{\bullet},d)].
\end{equation}
It is clear that $[(\mathfrak{F}^{\bullet},d)]=[(\tilde{\mathfrak{F}}^{\bullet},\tilde{d})]$ if $(\mathfrak{F}^{\bullet},d)$ and $(\tilde{\mathfrak{F}}^{\bullet},\tilde{d})$ are quasi-isomorphic.

For a Hermitian cohesive module $\mathcal{E}= (E^{\bullet}, A^{E^{\bullet}\prime\prime})$ on $X$, we know $\underline{F}_X(\mathcal{E})$ is a bounded complex on $X$ with coherent cohomologies. We then define the \emph{cycle} of $\mathcal{E}$ to be the current
\begin{equation} \label{eq: cycle of a Hermitian cohesive module}
[\mathcal{E}]:=[\underline{F}_X(\mathcal{E})].
\end{equation}

\subsubsection{Supertraces}
Let $E^{\bullet}$ be a bounded $\mathbb{Z}$-graded vector space over the base field $\mathbb{K}$. The \emph{supertrace}  is a map $\str:\End(E^{\bullet})\to \mathbb{K}$ defined by
\begin{equation}\label{eq: supertrace scalar}
\str(\phi):=\sum_l(-1)^l\text{Tr}(\phi|_{E^l}).
\end{equation}

Now let  $E^{\bullet}$ be a bounded $\mathbb{Z}$-graded complex vector bundle over a complex manifold $X$. We can extend the supertrace in \eqref{eq: supertrace scalar} to a map
\begin{equation}
\str: \Omega^{\bullet,\bullet}\otimes \End(E^{\bullet})\to  \Omega^{\bullet,\bullet}
\end{equation}
given by
\begin{equation}
\str(\omega\otimes \phi):=\omega \otimes \str(\phi).
\end{equation}

It is clear that $\str$ vanishes on supercommutators and is invariant under conjugations. See \cite[Section 4.2]{bismut2023coherent} for some details.

\subsubsection{$\dpar$-connections}
Let  $E^{\bullet}$ be a  $\mathbb{Z}$-graded complex vector bundle over a complex manifold $X$.  Recall that we have  $T_{\mathbb{C}}X=TX\bigoplus \overline{TX}$ where $TX$ and $\overline{TX}$ are the holomorphic and antiholomorphic tangent bundle, respectively. 

A \emph{$\dpar$-connection} on $E^{\bullet}$ is a map
\begin{equation}
\nabla^{E^{\bullet}\prime}: E^{\bullet}\to T^*X\times E^{\bullet}
\end{equation}
such that
\begin{equation}
\nabla^{E^{\bullet}\prime}(fe)=\dpar(f)e+f\nabla^{E^{\bullet}\prime}(e).
\end{equation}

If we also have a $\dbar$-connection $\nabla^{E^{\bullet}\prime\prime}$ on $E^{\bullet}$, then we can form the connection $\nabla^{E^{\bullet}}$ as
\begin{equation}\label{eq: nabla together}
\nabla^{E^{\bullet}}:=\nabla^{E^{\bullet}\prime}+\nabla^{E^{\bullet}\prime\prime}
\end{equation}

In general we do not impose any compatibility condition on $\nabla^{E^{\bullet}\prime}$ and $\nabla^{E^{\bullet}\prime\prime}$.

\subsection{A review of the main results  in \cite{larkang2021residue}}
Let us review the main results in \cite{larkang2021residue}

\begin{thm}[\cite{larkang2021residue} Theorem 1.1]\label{thm: larkang1.1}
Let 
\begin{equation}
(E^{\bullet}, v)=0\to E^{-N}\overset{v_{-N}}{\longrightarrow}\ldots \overset{v_{-1}}{\longrightarrow} E^0\to 0
\end{equation} be a bounded complex of Hermitian holomorphic vector bundles on $X$. If all its cohomologies $\mathfrak{H}^l(E^{\bullet}, v)$ have pure codimension $p\geq 1$ or vanish, and let $\nabla^{E^{\bullet}}$ be the connection on $\End(E^{\bullet})$  induced by an arbitrary $\dpar$-connection $\nabla^{E^{\bullet}\prime}$ and the known $\dbar$-connection $\nabla^{E^{\bullet}\prime\prime}$. Then we have the following equality of currents:
\begin{equation}\label{eq: Poincare-Lelong in Larkang1}
\frac{1}{(2\pi i)^p p!}\sum_{l=0}^{N-p}(-1)^l\text{Tr}(\nabla^{E^{\bullet}}(v_{-l-1})\ldots \nabla^{E^{\bullet}}(v_{-l-p}) R^E_{-l\to -l-p})=[(E^{\bullet}, v)],
\end{equation}
where $R^E$ is the residue current of $(E^{\bullet}, v)$.
\end{thm}

\begin{thm}[\cite{larkang2021residue} Theorem 1.2]\label{thm: larkang1.2}
Let $\mathfrak{F}$ be a coherent sheaf on $X$ of pure codimension $p$. Let $(E^{\bullet},v)$ be a Hermitian locally free resolution of $\mathfrak{F}$, and let $\nabla^{E^{\bullet}}$ be the connection on $\End(E^{\bullet})$  induced by an arbitrary $\dpar$-connection $\nabla^{E^{\bullet}\prime}$ and the known $\dbar$-connection $\nabla^{E^{\bullet}\prime\prime}$. Then we have the following equality of currents:
\begin{equation}\label{eq: Poincare-Lelong in Larkang2}
\frac{1}{(2\pi i)^p p!}\text{Tr}(\nabla^{E^{\bullet}}(v_{-1})\ldots \nabla^{E^{\bullet}}(v_{-p}) R^E_{0\to -p})=[\mathfrak{F}].
\end{equation}
\end{thm}
For the relation between \eqref{eq: Poincare-Lelong in Larkang2} and the classical Poincar\'{e}-Lelong formula
\begin{equation}
\frac{1}{2\pi i}\dbar\dpar \log|f|^2=[Z_f],
\end{equation}
 see \cite[Introduction]{larkang2021residue}.

Using the notation of supertrace, we can reformulate \eqref{eq: Poincare-Lelong in Larkang1}  as
\begin{equation}\label{eq: Poincare-Lelong in Larkang1 reformulate}
\frac{1}{(2\pi i)^p p!}\str((\nabla^{E^{\bullet}}(v))^p R^E)=[(E^{\bullet}, v)]
\end{equation}
and reformulate \eqref{eq: Poincare-Lelong in Larkang2}  as
\begin{equation}\label{eq: Poincare-Lelong in Larkang2 reformulate}
\frac{1}{(2\pi i)^p p!}\str((\nabla^{E^{\bullet}}(v))^p R^E)=[\mathfrak{F}]
\end{equation}
where $(\nabla^{E^{\bullet}}(v))^p$ denotes the composition of $\nabla^{E^{\bullet}}(v)$ for $p$ times.

Actually by Part 2 of Corollary \ref{coro: R vanish}, the only non-zero components on the left hand side of \eqref{eq: Poincare-Lelong in Larkang1 reformulate} are
$$
\text{Tr}(\nabla^{E^{\bullet}}(v_{-l-1})\ldots \nabla^{E^{\bullet}}(v_{-l-p}) R^E_{-l\to -l-p}) \text{, } 0\leq l \leq N-p,
$$
and the only non-zero component on the left hand side of \eqref{eq: Poincare-Lelong in Larkang2 reformulate} is 
$$
\text{Tr}(\nabla^{E^{\bullet}}(v_{-1})\ldots \nabla^{E^{\bullet}}(v_{-p}) R^E_{0\to -p}).
$$

\subsection{A generalized Poincar\'{e}-Lelong formula for cohesive modules}
In this subsection we state and prove the following theorem.

\begin{thm}\label{thm: Poincare-Lelong for cohesive modules}
Let $\mathcal{E}= (E^{\bullet}, A^{E^{\bullet}\prime\prime})$  be a Hermitian cohesive module on $X$ with
$$
A^{E^{\bullet}\prime\prime}=v_0+\nabla^{E^{\bullet}\prime\prime}+v_2+\ldots
$$
Let $R^{\mathcal{E}}$ be the residue current as in Definition \ref{defi: residue current of E}. Let $\nabla^{E^{\bullet}}$ be the connection on $\End(E^{\bullet})$  induced by an arbitrary $\dpar$-connection $\nabla^{E^{\bullet}\prime}$ and the known $\dbar$-connection $\nabla^{E^{\bullet}\prime\prime}$.

Let $\underline{F}_X(\mathcal{E})=(\mathfrak{E}^{\bullet},d)$ be the sheafification as Defined in Section \ref{subsection: cohesive and coherent}. If all its cohomologies $\mathfrak{H}^l(\mathfrak{E}^{\bullet},d)$ has pure codimension $p\geq 1$ or vanish, then we have the following equality of currents:
\begin{equation}\label{eq: Poincare-Lelong for cohesive modules 1}
\frac{1}{(2\pi i)^p p!}\str((\nabla^{E^{\bullet}}(v_0))^p R^{\mathcal{E}})=[\mathcal{E}]
\end{equation}
where $[\mathcal{E}]$ is given in \eqref{eq: cycle of a Hermitian cohesive module}.

In particular if $\mathfrak{F}$ is a coherent sheaf on $X$ with pure codimension $p\geq 1$. Let $\mathcal{E}= (E^{\bullet}, A^{E^{\bullet}\prime\prime})$  be a cohesive resolution of $\mathfrak{F}$ equipped with a Hermitian metric. Let  $R^{\mathcal{E}}$ and $\nabla^{E^{\bullet}}$ be as before, then we have the following equality of currents:
\begin{equation}\label{eq: Poincare-Lelong for cohesive modules 2}
\frac{1}{(2\pi i)^p p!}\str((\nabla^{E^{\bullet}}(v_0))^p R^{\mathcal{E}})=[\mathfrak{F}].
\end{equation}
\end{thm}
\begin{proof}
The strategy of the proof is to reduce to \eqref{eq: Poincare-Lelong in Larkang1 reformulate} via gauge equivalence and the comparison formula.

By Proposition \ref{prop: gauge equivalence to flat}, for any $x\in X$, there exists a small neighborhood $V$ of $x$ and a gauge equivalence
\begin{equation}
J: (E^{\bullet}, A^{E^{\bullet}\prime\prime})|_V\overset{\sim}{\to} (E^{\bullet}|_V, v_0+ \overline{\nabla}^{E^{\bullet}|_V\prime\prime})
\end{equation}
where $ (E^{\bullet}|_V, v_0+ \overline{\nabla}^{E^{\bullet}|_V\prime\prime})$ is a complex of holomorphic vector bundles with the same $v_0$.

Since the supertrace is invariant under conjugations, we know that
\begin{equation}\label{eq: conjugate by J}
\str((\nabla^{E^{\bullet}}(v_0))^p R^{\mathcal{E}}|_V)=\str\big((J\circ (\nabla^{E^{\bullet}}(v_0))\circ J^{-1})^p (J\circ R^{\mathcal{E}}|_V\circ J^{-1})\big)
\end{equation}

Let $R^{\overline{\mathcal{E}}|_V}$ be the residue current associated with the complex of holomorphic vector bundles $(E^{\bullet}|_V, v_0+ \overline{\nabla}^{E^{\bullet}|_V\prime\prime})$. By Corollary \ref{coro: homotopic comparison formula} we get
\begin{equation}\label{eq: JR J-1 first}
\begin{split}
J\circ &R^{\mathcal{E}}|_V\circ J^{-1}=R^{\overline{\mathcal{E}}|_V}+(v_0+ \overline{\nabla}^{E^{\bullet}|_V\prime\prime})(M^J \circ J^{-1})\\
&=R^{\overline{\mathcal{E}}|_V}+v_0(M^J \circ J^{-1})+\overline{\nabla}^{E^{\bullet}|_V\prime\prime}(M^J \circ J^{-1})
\end{split}
\end{equation}
where $M^J$ is the current associated with $J$ as in \eqref{eq: M phi}. Notice here the homotopy operator $\gamma$ in \eqref{eq: homotopic comparison formula} vanishes as $J\circ J^{-1}=\id$.

Since the cohomologies of $(E^{\bullet}|_V, v_0+ \overline{\nabla}^{E^{\bullet}|_V\prime\prime})$ have codimension $p$, by Corollary \ref{coro: R vanish} Part 2 and Corollary \ref{coro: vanishing of M} Part 2 we have
\begin{equation}
R^{\overline{\mathcal{E}}|_V}_{q\to k}=0 \text{ for } k\geq q-p+1
\end{equation}
and
\begin{equation}
M^J_{q\to k}=0 \text{ for } k\geq q-p.
\end{equation}
In other words
\begin{equation}
R^{\overline{\mathcal{E}}|_V}\in \Gamma(V,\mathcal{D}^{\bullet,\bullet}_X  \hat{\otimes}  \End^{\leq -p}( E^{\bullet}))
\end{equation}
and
\begin{equation}
M^J\in \Gamma(V,\mathcal{D}^{\bullet,\bullet}_X  \hat{\otimes}  \End^{\leq -p-1}( E^{\bullet})).
\end{equation}
Since $J^{-1}: (E^{\bullet}|_V, v_0+ \overline{\nabla}^{E^{\bullet}|_V\prime\prime})\to (E^{\bullet}, A^{E^{\bullet}\prime\prime})|_V$ is a degree $0$ morphism, its components preserve or lower the $E^{\bullet}$ degree. Hence 
\begin{equation}
M^J \circ J^{-1}\in \Gamma(V,\mathcal{D}^{\bullet,\bullet}_X  \hat{\otimes}  \End^{\leq -p-1}( E^{\bullet})).
\end{equation}
Since $v_0$ increases the $E^{\bullet}$ degree by $1$, and $\overline{\nabla}^{E^{\bullet}|_V\prime\prime}$ preserves the $E^{\bullet}$ degree,   we have
\begin{equation}
v_0(M^J \circ J^{-1})\in  \Gamma(V,\mathcal{D}^{\bullet,\bullet}_X  \hat{\otimes}  \End^{\leq -p}( E^{\bullet}))
\end{equation}
and
\begin{equation}
\overline{\nabla}^{E^{\bullet}|_V\prime\prime}(M^J \circ J^{-1})\in  \Gamma(V,\mathcal{D}^{\bullet,\bullet}_X  \hat{\otimes}  \End^{\leq -p-1}( E^{\bullet})).
\end{equation}
To simplify the notation, let us denote $\overline{\nabla}^{E^{\bullet}|_V\prime\prime}(M^J \circ J^{-1})$ by $\alpha$. Then \eqref{eq: JR J-1 first} becomes
\begin{equation}\label{eq: JR J-1 second}
J\circ R^{\mathcal{E}}|_V\circ J^{-1}=R^{\overline{\mathcal{E}}|_V}+v_0(M^J \circ J^{-1})+\alpha
\end{equation}
where 
\begin{equation}\label{eq: degree of alpha}
\begin{split}
R^{\overline{\mathcal{E}}|_V}&\in \Gamma(V,\mathcal{D}^{\bullet,\bullet}_X  \hat{\otimes}  \End^{\leq -p}( E^{\bullet})),\\
v_0(M^J \circ J^{-1})&\in \Gamma(V,\mathcal{D}^{\bullet,\bullet}_X  \hat{\otimes}  \End^{\leq -p}( E^{\bullet})),\\
\alpha &\in \Gamma(V,\mathcal{D}^{\bullet,\bullet}_X  \hat{\otimes}  \End^{\leq -p-1}( E^{\bullet})).
\end{split}
\end{equation}

Next we prove the following lemma on the term $(J\circ (\nabla^{E^{\bullet}}(v_0))\circ J^{-1})^p$.

\begin{lemma}\label{lemma: Jnabla v J -1 to p}
There exists another $\dpar$-connection $\overline{\nabla}^{E^{\bullet}\prime}$  hence  a connection $\overline{\nabla}^{E^{\bullet}}=\overline{\nabla}^{E^{\bullet}\prime}+\nabla^{E^{\bullet}\prime\prime}$ such that
\begin{equation}\label{eq: Jnabla v J -1 to p}
(J\circ (\nabla^{E^{\bullet}}(v_0))\circ J^{-1})^p=(\overline{\nabla}^{E^{\bullet}}(v_0))^p+\beta
\end{equation}
where 
\begin{equation}\label{eq: degree of beta}
\beta\in \Gamma(V,\Omega^{\bullet,\bullet}_X  \hat{\otimes}  \End^{\leq p-1}( E^{\bullet})).
\end{equation}
\end{lemma}
\begin{proof}[Proof of Lemma \ref{lemma: Jnabla v J -1 to p}]
By \eqref{eq: A flat decomposed} we have $\nabla^{E^{\bullet}\prime\prime}(v_0)=0$ hence 
\begin{equation}\label{eq: J nabla v J-1 first}
J\circ(\nabla^{E^{\bullet}}(v_0))\circ J^{-1}=J\circ (\nabla^{E^{\bullet}\prime}(v_0))\circ J^{-1}=(J\circ \nabla^{E^{\bullet}\prime}\circ J^{-1})(J\circ v_0\circ J^{-1}).
\end{equation}

As in \eqref{eq: decomposition of morphism} we decompose $J$ into
\begin{equation}
J=J_0+J_1+\ldots
\end{equation}
where
$$
J_i\in \Gamma(V,\Omega^{0,i}_X \hat{\otimes}  \End^{-i}( E^{\bullet})).
$$
In particular $J_0\in \Gamma(V, \End^{0}( E^{\bullet}))$ is invertible.
Similarly  we decompose $J^{-1}$ into
\begin{equation}
J^{-1}=(J_0)^{-1}+(J^{-1})_1+\ldots
\end{equation}
Notice that the $0$th term of $J^{-1}$ is $(J_0)^{-1}$.

Therefore we have
\begin{equation}\label{eq: J nabla J-1 first}
\begin{split}
J\circ \nabla^{E^{\bullet}\prime}&\circ J^{-1}=(J_0+J_{\geq 1})\circ \nabla^{E^{\bullet}\prime}\circ (J^{-1}_0+(J^{-1})_{\geq 1})\\
&=J_0\circ \nabla^{E^{\bullet}\prime}\circ J^{-1}_0 +J_{\geq 1}\circ\nabla^{E^{\bullet}\prime}\circ (J^{-1}_0+(J^{-1})_{\geq 1})\\
&+(J_0+J_{\geq 1})\circ\nabla^{E^{\bullet}\prime}\circ(J^{-1})_{\geq 1}
\end{split}
\end{equation}

$J_0\circ \nabla^{E^{\bullet}\prime}\circ J^{-1}_0$ is again a $\dpar$-connection, which we denote by $\overline{\nabla}^{E^{\bullet}\prime}$. Moreover the term
\begin{equation}
\begin{split}
&J_{\geq 1}\circ\nabla^{E^{\bullet}\prime}\circ (J^{-1}_0+(J^{-1})_{\geq 1})
+(J_0+J_{\geq 1})\circ\nabla^{E^{\bullet}\prime}\circ(J^{-1})_{\geq 1}\\
&\in \Gamma(V,\Omega^{\bullet,\bullet}_X  \hat{\otimes}  \End^{\leq -1}( E^{\bullet}))
\end{split}
\end{equation}
which we denote by $\beta_1$. Hence \eqref{eq: J nabla J-1 first} becomes
\begin{equation}\label{eq: J nabla J-1 second}
J\circ \nabla^{E^{\bullet}\prime}\circ J^{-1}=\overline{\nabla}^{E^{\bullet}\prime}+\beta_1
\end{equation}

On the other hand since $v_0$ is unchanged under conjugation by $J$, we know that
\begin{equation}\label{eq: J v J-1}
J\circ v_0\circ J^{-1}=v_0+\beta_2
\end{equation}
where $\beta_2\in \Gamma(V,\Omega^{\bullet,\bullet}_X  \hat{\otimes}  \End^{\leq 0}( E^{\bullet}))$. 

Combine \eqref{eq: J nabla v J-1 first}, \eqref{eq: J nabla J-1 second}, and \eqref{eq: J v J-1} we get
\begin{equation}\label{eq: J nabla v J-1 second}
\begin{split}
J\circ&(\nabla^{E^{\bullet}}(v_0))\circ J^{-1}=(\overline{\nabla}^{E^{\bullet}\prime}+\beta_1)(v_0+\beta_2)\\
&=\overline{\nabla}^{E^{\bullet}\prime}(v_0)+\beta_1(v_0)+\overline{\nabla}^{E^{\bullet}\prime}(\beta_1)+\beta_1(\beta_2).
\end{split}
\end{equation}
We know $\beta_1(v_0)+\overline{\nabla}^{E^{\bullet}\prime}(\beta_1)+\beta_1(\beta_1)\in \Gamma(V,\Omega^{\bullet,\bullet}_X  \hat{\otimes}  \End^{\leq 0}( E^{\bullet}))$, which we denote by $\beta_3$. 

\eqref{eq: J nabla v J-1 second} gives
\begin{equation}\label{eq: Jnabla vJ^-1 to p first}
(J\circ (\nabla^{E^{\bullet}}(v_0))\circ J^{-1})^p=(\overline{\nabla}^{E^{\bullet}\prime}(v_0)+\beta_3)^p=(\overline{\nabla}^{E^{\bullet}}(v_0)+\beta_3)^p.
\end{equation}
Since $\overline{\nabla}^{E^{\bullet}}(v_0)\in \Gamma(V,\Omega^{\bullet,\bullet}_X  \hat{\otimes}  \End^{\leq 1}( E^{\bullet}))$ and $\beta_3 \in \Gamma(V,\Omega^{\bullet,\bullet}_X  \hat{\otimes}  \End^{\leq 0}( E^{\bullet}))$, the expansion of the right hand side of \eqref{eq: Jnabla vJ^-1 to p first} gives \eqref{eq: Jnabla v J -1 to p}. We finish the proof of Lemma \ref{lemma: Jnabla v J -1 to p}.
\end{proof}

By \eqref{eq: conjugate by J}, \eqref{eq: JR J-1 second}, and \eqref{eq: Jnabla v J -1 to p} we have 
\begin{equation}\label{eq: conjugate by J expand}
\begin{split}
\str((\nabla^{E^{\bullet}}(v_0))^p R^{\mathcal{E}}|_V)&=\str\big( (\overline{\nabla}^{E^{\bullet}}(v_0))^p+\beta)(R^{\overline{\mathcal{E}}|_V}+v_0(M^J \circ J^{-1})+\alpha)\big)\\
&=\str( (\overline{\nabla}^{E^{\bullet}}(v_0))^p R^{\overline{\mathcal{E}}|_V})+\str\big( (\overline{\nabla}^{E^{\bullet}}(v_0))^p \big(v_0(M^J \circ J^{-1})\big)\big)\\
&+\str\big(\beta(R^{\overline{\mathcal{E}}|_V}+v_0(M^J \circ J^{-1})+\alpha)+(\overline{\nabla}^{E^{\bullet}}(v_0))^p\alpha\big).
\end{split}
\end{equation}
By \eqref{eq: degree of alpha} and \eqref{eq: degree of beta}, and the fact that 
$$
(\overline{\nabla}^{E^{\bullet}}(v_0))^p\in  C^{\infty}(V,\Omega^{\bullet,\bullet}_X  \hat{\otimes}  \End^{p}( E^{\bullet})),
$$
we know that 
\begin{equation}
\beta(R^{\overline{\mathcal{E}}|_V}+v_0(M^J \circ J^{-1})+\alpha)+(\overline{\nabla}^{E^{\bullet}}(v_0))^p\alpha \in  \Gamma(V,\mathcal{D}^{\bullet,\bullet}_X  \hat{\otimes}  \End^{\leq -1}( E^{\bullet}))
\end{equation}
hence its supertrace vanishes by degree reason. Therefore \eqref{eq: conjugate by J expand} gives
\begin{equation}\label{eq: conjugate by J two term}
\str((\nabla^{E^{\bullet}}(v_0))^p R^{\mathcal{E}}|_V)=\str( (\overline{\nabla}^{E^{\bullet}}(v_0))^p R^{\overline{\mathcal{E}}|_V})+\str\big( (\overline{\nabla}^{E^{\bullet}}(v_0))^p \big(v_0(M^J \circ J^{-1})\big)\big).
\end{equation}

We can prove that $\str\big( (\overline{\nabla}^{E^{\bullet}}(v_0))^p \big(v_0(M^J \circ J^{-1})\big)\big)$ also vanishes. Actually by definition 
\begin{equation}
v_0(M^J \circ J^{-1})=[v_0, M^J \circ J^{-1}]=v_0\circ M^J\circ J^{-1}+M^J\circ J^{-1}\circ v_0,
\end{equation}
where $[\cdot, \cdot]$ denotes the supercommutator.
By the same argument as in the proof of Lemma \ref{lemma: A sigma sigma=sigma A sigma} we have
\begin{equation}
(\overline{\nabla}^{E^{\bullet}}(v_0))^p\circ v_0=v_0\circ (\overline{\nabla}^{E^{\bullet}}(v_0))^p.
\end{equation}
Therefore
\begin{equation}
\begin{split}
(\overline{\nabla}^{E^{\bullet}}&(v_0))^p \big(v_0(M^J \circ J^{-1})\big)\\
=&(\overline{\nabla}^{E^{\bullet}}(v_0))^p \circ v_0\circ (M^J \circ J^{-1})
+(\overline{\nabla}^{E^{\bullet}}(v_0))^p \circ (M^J \circ J^{-1})\circ v_0\\
=&v_0\circ (\overline{\nabla}^{E^{\bullet}}(v_0))^p \circ (M^J \circ J^{-1})+(\overline{\nabla}^{E^{\bullet}}(v_0))^p \circ (M^J \circ J^{-1})\circ v_0\\
=&[v_0, (\overline{\nabla}^{E^{\bullet}}(v_0))^p \circ (M^J \circ J^{-1})]
\end{split}
\end{equation}
whose supertrace vanishes since supertrace vanishes on supercommutators. Therefore \eqref{eq: conjugate by J two term} gives
\begin{equation}\label{eq: conjugate by J one term}
\str((\nabla^{E^{\bullet}}(v_0))^p R^{\mathcal{E}}|_V)=\str( (\overline{\nabla}^{E^{\bullet}}(v_0))^p R^{\overline{\mathcal{E}}|_V}).
\end{equation}
Since  $(E^{\bullet}|_V, v_0+ \overline{\nabla}^{E^{\bullet}|_V\prime\prime})$ is a complex of holomorphic vector bundles on $V$,  by \eqref{eq: Poincare-Lelong in Larkang1 reformulate} we have 
\begin{equation}\label{eq: local Poincare-Lelong bundle case}
\frac{1}{(2\pi i)^p p!}\str( (\overline{\nabla}^{E^{\bullet}}(v_0))^p R^{\overline{\mathcal{E}}|_V})=[(E^{\bullet}|_V, v_0+ \overline{\nabla}^{E^{\bullet}|_V\prime\prime})].
\end{equation}
We know that 
\begin{equation}
[\mathcal{E}]\cap V=[(E^{\bullet}|_V, v_0+ \overline{\nabla}^{E^{\bullet}|_V\prime\prime})]
\end{equation}
since $J$ induces a quasi-isomorphism on the complex of sheaves. From \eqref{eq: conjugate by J one term} and \eqref{eq: local Poincare-Lelong bundle case} we know
\begin{equation}
\frac{1}{(2\pi i)^p p!}\str( (\overline{\nabla}^{E^{\bullet}}(v_0))^p R^{\overline{\mathcal{E}}|_V})=[\mathcal{E}]\cap V
\end{equation}
for any sufficiently small open neighborhood $V$ of $x\in X$. We thus get \eqref{eq: Poincare-Lelong for cohesive modules 1}. The proof of \eqref{eq: Poincare-Lelong for cohesive modules 1} is the same.
\end{proof}

We have the following result on the non-pure codimension case.

\begin{coro}\label{coro: Poincare-Lelong non-pure}
Let $\mathcal{E}= (E^{\bullet}, A^{E^{\bullet}\prime\prime})$  be a Hermitian cohesive module on $X$. Let $\underline{F}_X(\mathcal{E})=(\mathfrak{E}^{\bullet},d)$ be the sheafification as Defined in Section \ref{subsection: cohesive and coherent}. If all its cohomologies $\mathfrak{H}^l(\mathfrak{E}^{\bullet},d)$ has  codimension $p\geq 1$ or vanish, then we have the following equality of currents:
\begin{equation}\label{eq: Poincare-Lelong for cohesive modules nonpure 1}
\frac{1}{(2\pi i)^p p!}\str((\nabla^{E^{\bullet}}(v_0))^p R^{\mathcal{E}})=[\mathcal{E}]_p
\end{equation}
where $[\mathcal{E}]_p$ is the sum over codimension $p$ components of $[\mathcal{E}]$.

In particular if $\mathfrak{F}$ is a coherent sheaf  with  codimension $p\geq 1$. Let $\mathcal{E}= (E^{\bullet}, A^{E^{\bullet}\prime\prime})$  be a cohesive resolution of $\mathfrak{F}$ equipped with a Hermitian metric. Then we have the following equality of currents:
\begin{equation}\label{eq: Poincare-Lelong for cohesive modules nonpure 2}
\frac{1}{(2\pi i)^p p!}\str((\nabla^{E^{\bullet}}(v_0))^p R^{\mathcal{E}})=[\mathfrak{F}]_p.
\end{equation}
\end{coro}
\begin{proof}
Let $W$ be the union of the components of $\supp\mathcal{E}$ with codimension $\geq p+1$. Then $W$ is a subvariety of codimension $\geq p+1$ in $X$. 

Since $\mathcal{E}|_{X\backslash W}$ has pure codimension $p$, we can  apply Theorem \ref{thm: Poincare-Lelong for cohesive modules} to $X\backslash W$ and get
\begin{equation}
\frac{1}{(2\pi i)^p p!}\str((\nabla^{E^{\bullet}}(v_0))^p R^{\mathcal{E}})|_{X\backslash W}=[\mathcal{E}]_p\cap  (X\backslash W).
\end{equation}

By Proposition \ref{prop: pseudomeromorphic currents closed under multiplication}, Proposition \ref{prop: [Z] is pseudomeromorphic}, and Definition \ref{defi: residue current of E}, both $\str((\nabla^{E^{\bullet}}(v_0))^p R^{\mathcal{E}})$ and  $[\mathcal{E}]_p$ are $(p,p)$-\emph{pseudomeromorphic} current on $X$. As $W$ has codimension $\geq p+1$, we have \eqref{eq: Poincare-Lelong for cohesive modules nonpure 1} by the \emph{dimension principle} given in Proposition \ref{prop:dimPrinciple}.  The proof of \eqref{eq: Poincare-Lelong for cohesive modules nonpure 2} is the same.
\end{proof}

\bibliography{residue}{}

\begin{thebibliography}{{Sta}24}

\bibitem[And05]{andersson2005residues}
Mats Andersson.
\newblock Residues of holomorphic sections and {L}elong currents.
\newblock {\em Ark. Mat.}, 43(2):201--219, 2005.

\bibitem[AW07]{andersson2007residue}
Mats Andersson and Elizabeth Wulcan.
\newblock Residue currents with prescribed annihilator ideals.
\newblock {\em Ann. Sci. \'{E}cole Norm. Sup. (4)}, 40(6):985--1007, 2007.

\bibitem[AW10]{andersson2010decomposition}
Mats Andersson and Elizabeth Wulcan.
\newblock Decomposition of residue currents.
\newblock {\em J. Reine Angew. Math.}, 638:103--118, 2010.

\bibitem[AW18]{andersson2018direct}
Mats Andersson and Elizabeth Wulcan.
\newblock Direct images of semi-meromorphic currents.
\newblock {\em Ann. Inst. Fourier (Grenoble)}, 68(2):875--900, 2018.

\bibitem[Blo10]{block2010duality}
Jonathan Block.
\newblock Duality and equivalence of module categories in noncommutative
  geometry.
\newblock In {\em A celebration of the mathematical legacy of {R}aoul {B}ott},
  volume~50 of {\em CRM Proc. Lecture Notes}, pages 311--339. Amer. Math. Soc.,
  Providence, RI, 2010.

\bibitem[BSW23]{bismut2023coherent}
Jean-Michel Bismut, Shu Shen, and Zhaoting Wei.
\newblock {\em Coherent {S}heaves, {S}uperconnections, and
  {R}iemann-{R}och-{G}rothendieck}, volume 347 of {\em Progress in
  Mathematics}.
\newblock Birkh\"{a}user/Springer, Cham, 2023.

\bibitem[CH78]{herrera1978courants}
Nicolas~R. Coleff and Miguel~E. Herrera.
\newblock {\em Les courants r\'esiduels associ\'es \`a{} une forme
  m\'eromorphe}, volume 633 of {\em Lecture Notes in Mathematics}.
\newblock Springer, Berlin, 1978.

\bibitem[CHL21]{chuang2021maurer}
Joseph Chuang, Julian Holstein, and Andrey Lazarev.
\newblock Maurer-{C}artan moduli and theorems of {R}iemann-{H}ilbert type.
\newblock {\em Appl. Categ. Structures}, 29(4):685--728, 2021.

\bibitem[Eis95]{eisenbud1995commutative}
David Eisenbud.
\newblock {\em Commutative algebra}, volume 150 of {\em Graduate Texts in
  Mathematics}.
\newblock Springer-Verlag, New York, 1995.
\newblock With a view toward algebraic geometry.

\bibitem[GH94]{griffiths1994principles}
Phillip Griffiths and Joseph Harris.
\newblock {\em Principles of algebraic geometry}.
\newblock Wiley Classics Library. John Wiley \& Sons, Inc., New York, 1994.
\newblock Reprint of the 1978 original.

\bibitem[GR84]{grauert1984coherent}
Hans Grauert and Reinhold Remmert.
\newblock {\em Coherent analytic sheaves}, volume 265 of {\em Grundlehren der
  mathematischen Wissenschaften [Fundamental Principles of Mathematical
  Sciences]}.
\newblock Springer-Verlag, Berlin, 1984.

\bibitem[Han24]{han2024characteristic}
Zhaobo Han.
\newblock Characteristic currents on cohesive modules.
\newblock {\em arXiv preprint arXiv:2404.09439}, 2024.

\bibitem[HL71]{herrera1971residues}
M.~Herrera and D.~Lieberman.
\newblock Residues and principal values on complex spaces.
\newblock {\em Math. Ann.}, 194:259--294, 1971.

\bibitem[JL21]{johansson2021explicit}
Jimmy Johansson and Richard L\"ark\"ang.
\newblock An explicit isomorphism of different representations of the {E}xt
  functor using residue currents.
\newblock {\em arXiv preprint arXiv:2109.00480}, 2021.

\bibitem[Joh23]{Johansson2023residue}
Jimmy Johansson.
\newblock A residue current associated with a twisting cochain: duality and
  comparison formula.
\newblock {\em arXiv preprint arXiv:2306.02458}, 2023.

\bibitem[L\"19]{larkang2019comparison}
Richard L\"ark\"ang.
\newblock A comparison formula for residue currents.
\newblock {\em Math. Scand.}, 125(1):39--66, 2019.

\bibitem[LW18]{larkang2018residue}
Richard L\"ark\"ang and Elizabeth Wulcan.
\newblock Residue currents and fundamental cycles.
\newblock {\em Indiana Univ. Math. J.}, 67(3):1085--1114, 2018.

\bibitem[LW21]{larkang2021residue}
Richard L\"ark\"ang and Elizabeth Wulcan.
\newblock Residue currents and cycles of complexes of vector bundles.
\newblock {\em Ann. Fac. Sci. Toulouse Math. (6)}, 30(5):961--984, 2021.

\bibitem[LW22]{larkang2022chern}
Richard L\"ark\"ang and Elizabeth Wulcan.
\newblock Chern currents of coherent sheaves.
\newblock {\em \'Epijournal G\'eom. Alg\'ebrique}, 6:Art. 14, 26, 2022.

\bibitem[{Sta}24]{stacks-project}
The {Stacks project authors}.
\newblock The stacks project.
\newblock \url{https://stacks.math.columbia.edu}, 2024.

\bibitem[TT78]{toledo1978duality}
Domingo Toledo and Yue Lin~L. Tong.
\newblock Duality and intersection theory in complex manifolds. {I}.
\newblock {\em Math. Ann.}, 237(1):41--77, 1978.

\bibitem[Voi02]{voisin2002counterexample}
Claire Voisin.
\newblock A counterexample to the {H}odge conjecture extended to {K}\"ahler
  varieties.
\newblock {\em Int. Math. Res. Not.}, (20):1057--1075, 2002.

\end{thebibliography}
\bibliographystyle{alpha}

\end{document}